\documentclass[psamsfonts]{amsart}
\usepackage{amssymb}
\usepackage{amscd}
\usepackage{mathtools}
\usepackage{amscd, enumerate}
\usepackage{tikz-cd}
\usepackage[utf8]{inputenc}
\usepackage[all]{xy}
\usepackage{hyperref}
\numberwithin{equation}{section}
\usetikzlibrary{matrix}
\makeatletter
\newcommand{\dhocolim@}[2]{%
  \vtop{\m@th\ialign{##\cr
    \hfil$#1\operator@font hocolim$\hfil\cr
    \noalign{\nointerlineskip\kern1.5\ex@}#2\cr
    \noalign{\nointerlineskip\kern-\ex@}\cr}}%
}

\newcommand{\dhocolim}{%
  \mathop{\mathpalette\dhocolim@{\rightarrowfill@\textstyle}}\nmlimits@
}

\newtheorem{thm}{Theorem}[section]
\newtheorem{prop}[thm]{Proposition}
\newtheorem{lem}[thm]{Lemma}
\newtheorem{cor}[thm]{Corollary}
\newtheorem{clm}{Claim}

\theoremstyle{definition}

\newtheorem{conv}[thm]{Convention}

\newtheorem{rem}[thm]{Remark}

\theoremstyle{remark}

\DeclareMathOperator*{\ModR}{Mod-R}

\DeclareMathOperator{\Hom}{Hom}
\DeclareMathOperator{\RHom}{\mathbf{R}Hom}

\DeclareMathOperator*{\Ann}{Ann}
\DeclareMathOperator*{\Ass}{Ass}
\DeclareMathOperator*{\Ker}{Ker}

\DeclareMathOperator*{\Spec}{Spec}

\DeclareMathOperator{\hocolim}{hocolim}

\DeclareMathOperator{\colim}{colim}
\DeclareMathOperator{\Mcolim}{Mcolim}
\DeclareMathOperator{\Mlim}{Mlim}
\DeclareMathOperator{\aisle}{aisle}

\DeclareMathOperator{\projR}{proj-R}

\DeclareMathOperator*{\Supp}{Supp}

\DeclareMathOperator*{\Der}{\mathbf{D}}
\DeclareMathOperator*{\Htpy}{\mathbf{K}}
\DeclareMathOperator*{\Loc}{Loc}

\DeclareMathOperator*{\Cone}{Cone}

\DeclareMathOperator*{\pp}{\mathfrak{p}}

\DeclareMathOperator*{\pa}{aisle}

\newenvironment{cprf}
{\begin{proof}[Proof of Claim~\theclm]}
{ 
\end{proof}}

\newcommand{\EMP}{\textbf}
\newcommand*{\Perp}[1]{{}^{\perp_{#1}}}

\begin{document}
\title[Compactly generated t-structures over a commutative ring]{Compactly generated t-structures in the derived category of a commutative ring}
\author{Michal Hrbek}
\address{Institute of Mathematics CAS, \v{Z}itn\'{a} 25, 115 67 Prague, Czech Republic}
\email{hrbek@math.cas.cz}

\thanks{M. Hrbek was supported by the Czech Academy of Sciences Programme for research and mobility support of starting researchers, project MSM100191801.} 

\begin{abstract}
		We classify all compactly generated t-structures in the unbounded derived category of an arbitrary commutative ring, generalizing the result of \cite{AJS} for noetherian rings. More specifically, we establish a bijective correspondence between the compactly generated t-structures and infinite filtrations of the Zariski spectrum by Thomason subsets. Moreover, we show that in the case of a commutative noetherian ring, any bounded below homotopically smashing t-structure is compactly generated. As a consequence, all cosilting complexes are classified up to equivalence.
\end{abstract}

\maketitle

\setcounter{tocdepth}{1} \tableofcontents

\section{Introduction}
There is a large supply of classification results for various subcategories of the unbounded derived category $\Der(R)$ of a commutative noetherian ring $R$. Since the work of Hopkins \cite{Ho} and Neeman \cite{N}, we know that the localizing subcategories of $\Der(R)$ are parametrized by data of a geometrical nature --- the subsets of the Zariski spectrum $\Spec(R)$. As a consequence, the famous telescope conjecture, stating that any smashing localizing subcategory is generated by a set of compact objects, holds for the category $\Der(R)$. These compactly generated localizations then correspond to those subsets of $\Spec(R)$, which are specialization closed, that is, upper subsets in the poset $(\Spec(R), \subseteq)$.

A more recent work \cite{AJS} provides a ``semistable'' version of the latter classification result. Specifically, it establishes a bijection between compactly generated t-structures and infinite decreasing sequences of specialization subsets of $\Spec(R)$ indexed by the integers. The concept of a t-structure in a triangulated category was introduced by Be\u{\i}linson, Bernstein, and Deligne \cite{BBD}, and can be seen as a way of constructing abelian categories inside triangulated categories, together with cohomological functors. In the setting of the derived category of a ring, t-structures proved to be an indispensable tool in various general instances of tilting theory, replacing the role played by the ordinary torsion pairs in more traditional tilting frameworks (see e.g. \cite{PV}, \cite{NSZ}, \cite{AMV}, and \cite{MV}).

Not all the mentioned results carry well to the generality of an arbitrary commutative ring $R$ (that is, without the noetherian assumption). Namely, the classification of all localizing subcategories via geometrical invariants is hopeless (see \cite{N2}), and the telescope conjecture does not hold in general, the first counterexample of this is due to Keller \cite{Kel2}. However, when restricting to the subcategories induced by compact objects, the situation is far more optimistic. The prime example of this is the classification of compact localizations due to Thomason \cite{T}. Another piece of evidence is provided by the recent classification of $n$-tilting modules \cite{HS}, extending the previous work \cite{APST} from the noetherian rings to general commutative ones. In both cases, the statement of the general result is obtained straightforwardly, by replacing any occurrence of ``a specialization closed subset of $\Spec(R)$'' by ``a Thomason subset of $\Spec(R)$''. The methods of the proof however differ substantially, due to the fact that lot of machinery available in the noetherian world simply does not work in the general situation.

In the present paper, we continue in this path by proving the following two theorems:
\begin{enumerate}
		\item \EMP{(Theorem~\ref{T:full})} Let $R$ be an arbitrary commutative ring. Then the compactly generated t-structures in $\Der(R)$ are in a bijective correspondence with decreasing sequences 
				$$\cdots \supseteq X_{n-1} \supseteq X_n \supseteq X_{n+1} \supseteq \cdots$$
				of Thomason subsets of the Zariski spectrum $\Spec(R)$.
		\item \EMP{(Theorem~\ref{T:TC})} Let $R$ be a commutative noetherian ring. Then any bounded below homotopically smashing t-structure is compactly generated.
\end{enumerate}

We postpone the precise formulation of both statements, as well as the relevant definitions, to the body of the paper. The first of the two results is a generalization of \cite[Theorem 3.10]{AJS} --- from the commutative noetherian ring to an arbitrary commutative ring. Since many of the methods used in \cite{AJS} are specific to the noetherian situation, we were forced to use different techniques. First, we characterize the compact generation of a bounded below t-structure using injective envelopes, see Lemma~\ref{L:boundedbelow} and Theorem~\ref{T:bounded}. This approach is also crucial in proving the second result (2), which can be seen as a version of the telescope conjecture for commutative noetherian rings adjusted for bounded below t-structures, as opposed to localizing pairs. As a corollary, any t-structure induced by a (bounded) cosilting complex is compactly generated (this is Corollary~\ref{C02}), extending the result proved in \cite{APST}, which established the cofinite type for $n$-cotilting modules over a commutative noetherian ring. As a consequence, cosilting complexes over commutative noetherian rings are classified in Theorem~\ref{T:cosilt}, generalizing the result for cosilting modules \cite[Theorem 5.1]{AH}. 

Finally, in Theorem~\ref{T:full} we remove the bounded below assumption and classify all compactly generated t-structures over a commutative ring. Even though we are not able to show in general, unlike in the bounded below case, that the aisle of such a t-structure is determined by supports cohomology-wise, we show that these t-structures are always generated by suspensions of Koszul complexes, using similar techniques as in the ``stable'' version of this classification for localizing pairs from \cite{KP}. 
\subsection*{Basic notation}
We work in the unbounded derived category $\Der(R) := \Der(\ModR)$ of the category $\ModR$ of all $R$-modules over a commutative ring $R$; we refer to \cite[\S 13 and \S 14]{KP2} for a sufficiently up-to-date exposition. Whenever talking about subcategories, we mean full subcategories. Given a subcategory $\mathcal{C} \subseteq \Der(R)$, and a set $I \subseteq \mathbb{Z}$, we let
$$\mathcal{C}\Perp{I}=\{X \in \Der(R) \mid \Hom_{\Der(R)}(C,X[n])=0 ~\forall C \in \mathcal{C}, n \in I\},$$
and
$$\Perp{I}\mathcal{C}=\{X \in \Der(R) \mid \Hom_{\Der(R)}(X,C[n])=0 ~\forall C \in \mathcal{C}, n \in I\}.$$
Usually, the role of $I$ will be played by symbols $k$, $\geq k$, $\leq k$, $<k$, $>k$ for some integer $k$ with their obvious interpretation as subsets of $\mathbb{Z}$.

The symbol $\Der(R)^c$ stands for the thick subcategory of all compact objects of $\Der(R)$. Recall that the compact objects of $\Der(R)$ are, up to a quasi-isomorphism, precisely the perfect complexes, that is, bounded complexes of finitely generated projective modules.

Complexes are written using the cohomological notation, that is, a complex $X$ has coordinates $X^n$, with the degree $n \in \mathbb{Z}$ increasing along the differentials $d_X^n: X^n \rightarrow X^{n+1}$. We denote by $B^n(X), Z^n(X)$, and $H^n(X)$ the $n$-th coboundary, cocycle, and cohomology of the complex $X$. For any complex $X \in \Der(R)$, we define its cohomological infimum
$$\inf(X) = \inf \{n \in \mathbb{Z} \mid H^n(X) \neq 0\},$$
and recall the usual notation $\Der^+(R) = \{X \in \Der(R) \mid \inf(X) \in \mathbb{Z}\}$ for the full subcategory of cohomologically bounded below complexes. The supremum $\sup(X)$ of a complex is defined dually.

We will freely use the calculus of the total derived bifunctors $\RHom_R(-,-)$ and $- \otimes_R^\mathbf{L} -$, considered as functors $\Der(R) \times \Der(R)^{} \rightarrow \Der(R)$ and $\Der(R) \times \Der(R^{op}) \rightarrow \Der(R)$, respectively.
Finally, given a subcategory $\mathcal{C}$ of $\ModR$ we let $\Htpy(\mathcal{C})$ denote the homotopy category of all complexes with coordinates in $\mathcal{C}$. If $\mathcal{C} = \ModR$, we write just $\Htpy(R)$. Also, we consider the bounded variants $\Htpy^\#(\mathcal{C})$ of homotopy category, where $\#$ can be one of symbols $\{< 0, > 0,-,+,b\}$ with their usual interpretation.
\section{Torsion pairs in the derived category}
	A pair $(\mathcal{U},\mathcal{V})$ of subcategories of $\Der(R)$ is called a \EMP{torsion pair} provided that 
\begin{enumerate}
	\item[(i)] $\mathcal{U}=\Perp{0}\mathcal{V}$ and $\mathcal{V}=\mathcal{U}\Perp{0}$, 
	\item[(ii)] both $\mathcal{U}$ and $\mathcal{V}$ are closed under direct summands, and
	\item[(iii)] for any $X \in \Der(R)$ there is a triangle
\begin{equation}\label{approxtriangle}
		U \rightarrow X \rightarrow V \xrightarrow{} U[1]
\end{equation}
with $U \in \mathcal{U}$ and $V \in \mathcal{V}$.
\end{enumerate}
A torsion pair $(\mathcal{U},\mathcal{V})$ in $\Der(R)$ is called:
\begin{itemize}
	\item a \EMP{t-structure} if $\mathcal{U}[1] \subseteq \mathcal{U}$,
	\item a \EMP{co-t-structure} if $\mathcal{U}[-1] \subseteq \mathcal{U}$\footnote{These also appear under the name ``weight structures'' in the literature.},
	\item a \EMP{localizing pair} if $\mathcal{U}[\pm 1] \subseteq \mathcal{U}$, that is, when both $\mathcal{U}$ and $\mathcal{V}$ are triangulated subcategories of $\Der(R)$.
\end{itemize}
		In this paper, we will be mostly interested in t-structures --- in this setting we call the subcategory $\mathcal{U}$ (resp. $\mathcal{V}$) the \EMP{aisle} (resp. the \EMP{coaisle}) of the t-structure $(\mathcal{U},\mathcal{V})$.  In the t-structure case, the \EMP{approximation} triangle \ref{approxtriangle} is unique up to a unique isomorphism, and it is of the form
		$$\tau_\mathcal{U} X \rightarrow X \rightarrow \tau_\mathcal{V} X \xrightarrow{}\tau_\mathcal{U}[1],$$
		where $\tau_\mathcal{U}$ (resp. $\tau_\mathcal{V}$) is the right (resp. left) adjoint to the inclusion of the aisle (resp. coaisle):
		$$
	\begin{tikzcd}
			\arrow[bend right]{r}[below]{\subseteq}\mathcal{U}  & \Der(R) \arrow[bend right]{l}[above]{\tau_\mathcal{U}} \arrow[bend right]{r}[below]{\tau_\mathcal{V}} & \mathcal{V}  \arrow[bend right]{l}[above]{\subseteq}.
	\end{tikzcd}
		$$
\subsection{Standard (co-)t-structures and truncations}
The pair of subcategories $(\Der^{\leq 0},\Der^{>0})$, where
$$\Der{}^{\leq 0} = \{X \in \Der(R) \mid \sup(X) \leq 0\},$$
and
$$\Der{}^{>0}= \{X \in \Der(R) \mid \inf(X) > 0\},$$
form the so-called \EMP{standard t-structure}. The left and right approximation with respect to this t-structure are denoted simply by $\tau^{\leq 0}$ and $\tau^{>0}$, and they can be identified with the standard ``soft'' truncations of complexes. Of course, we also adopt the notation $\Der^{\leq n}, \Der^{> n}, \tau^{\leq n}, \tau^{>n}$ for the obvious versions of this t-structure shifted by $n \in \mathbb{Z}$, and the associated soft truncation functors in degree $n$.

On the side of co-t-structures, we follow \cite[Example 2.9]{AMV0} and define a \EMP{standard co-t-structure} $(\Htpy_p^{>0}, \Der^{\leq 0})$, where $\Htpy_p^{>0}$ is the subcategory of all complexes from $\Der(R)$ quasi-isomorphic to a K-projective complex with zero non-positive components. Since $\Der(R)$ is equivalent to the homotopy category $\Htpy_p$ of all K-projective complexes, restricting to $\Htpy_p^{\leq 0} \simeq \Der^{\leq 0}$, it follows that $\Htpy_p^{>0} = \Perp{0}\Der^{\leq 0}$. For any complex $X$, one can obtain an approximation triangle with respect to this co-t-structure using the brutal truncations (also called the stupid truncations). Denote by $\sigma^{>n}$, and $\sigma^{\leq n}$ the right and left brutal truncations at degree $n$. Let $P$ be a K-projective complex quasi-isomorphic to $X$, then there is a triangle
$$\sigma^{>0}P \rightarrow X \rightarrow \sigma^{\leq 0} P \xrightarrow{} \sigma^{>0}P[1],$$
and it approximates $X$ with respect to the standard co-t-structure. 
\subsection{Generation of torsion pairs}
Let $(\mathcal{U},\mathcal{V})$ be a torsion pair,  and $\mathcal{S}$ be a subclass of $\Der(R)$. We say that the torsion pair $(\mathcal{U},\mathcal{V})$ is \EMP{generated by the class $\mathcal{S}$} if $(\mathcal{U},\mathcal{V})=(\Perp{0}(\mathcal{S}\Perp{0}),\mathcal{S}\Perp{0})$. If $\mathcal{S}$ consists of objects from $\Der(R)^c$, we say that $(\mathcal{U},\mathcal{V})$ is \EMP{compactly generated}.
Dually, the torsion pair $(\mathcal{U},\mathcal{V})$ is \EMP{cogenerated by $\mathcal{S}$} if $(\mathcal{U},\mathcal{V})=(\Perp{0}\mathcal{S},(\Perp{0}\mathcal{S})\Perp{0})$. 

	Given a general class $\mathcal{S}$, we cannot always claim that the pair $(\Perp{0}(\mathcal{S}\Perp{0}),\mathcal{S}\Perp{0})$ is a torsion pair. But if $\mathcal{S}$ forms a set, we can always use it to generate a t-structure. In particular, since $\Der^c(R)$ is a skeletally small category, we can always generate a t-structure by a given family of compact objects.
	\begin{thm}\emph{(\cite[Proposition 3.2]{AJS2})}\label{aisles}
		If $\mathcal{C}$ is a \EMP{set} of objects of $\Der(R)$, then $\Perp{0}(\mathcal{C}\Perp{\leq 0})$ is an aisle of a t-structure.
\end{thm}
		If $\mathcal{C}$ is a set, we let $\aisle(\mathcal{C})$ denote the class $\Perp{0}(\mathcal{C}\Perp{\leq 0})$. By \cite[p. 6]{AJS}, $\aisle(\mathcal{C})$ coincides with the smallest subcategory of $\Der(R)$ containing $\mathcal{C}$ and closed under suspensions, extensions, and coproducts. 

\subsection{Homotopy (co)limits} It is well-known that, in general, the categorical (co)limit constructions have a very limited use in triangulated categories, which can be illustrated by the fact that any monomorphism, as well as any epimorphism, is split. The way to circumvent this shortcoming is to use \EMP{homotopy (co)limits} instead. Especially the directed versions of the homotopy colimits will be essential in our efforts. Because in many sources the emphasis is put only on directed systems of shape $\omega$, we feel it is necessary to recall the relevant facts on homotopy colimits of arbitrary shapes.

We start with a construction which can be, up to quasi-isomorphism, seen as a special case of the general homotopy (co)limit construction. For details, we refer the reader to \cite{NB}. Let 
$$X_0 \xrightarrow{f_0} X_1 \xrightarrow{f_1} X_2 \xrightarrow{f_2} \cdots$$
be a tower of morphism in $\Der(R)$, indexed by natural numbers. The \EMP{Milnor colimit} of this tower, denoted $\Mcolim_{n \geq 0}X_n$, is defined by the following triangle
$$\bigoplus_{n \geq 0} X_n \xrightarrow{1-(f_n)} \bigoplus_{n \geq 0}X_n \rightarrow \Mcolim_{n \geq 0}X_n \xrightarrow{}\bigoplus_{n \geq 0} X_n[1].$$
Dually, if 
$$\cdots \xrightarrow{f_2} X_2 \xrightarrow{f_1} X_1 \xrightarrow{f_0} X_0$$
is an inverse tower of morphisms in $\Der(R)$, the \EMP{Milnor limit} is defined as an object fitting into the triangle
$$\prod_{n \geq 0}X_n[-1] \rightarrow \Mlim_{n \geq 0}X_n \rightarrow \prod_{n \geq 0} X_n \xrightarrow{1-(f_n)} \prod_{n \geq 0}X_n.$$
Directly from the definition, one can check the quasi-isomorphism 
\begin{equation}\label{E:rhomlim}
\RHom_R(\Mcolim_{n \geq 0}X_n,Y) \simeq \Mlim_{n \geq 0} \RHom_R(X_n,Y)\end{equation}
 for any $Y \in \Der(R)$.
Milnor limits and colimits allow to reconstruct any complex $X$ from its brutal truncations, that is, there are isomorphisms in $\Der(R)$:
\begin{equation}\label{E:brutal}
X \simeq \Mcolim_{n \geq 0} \sigma^{>-n} X \text{   and   } X \simeq \Mlim_{n \geq 0} \sigma^{<n} X,
\end{equation}
where the maps in the towers are the natural ones (see e.g. \cite[Lemma 6.3, Theorem 10.2]{Kel}).

Now we address the more general notion (in a sense) of a homotopy colimit. The definition of a homotopy colimit can vary a lot depending on how general categories one considers. In our case of the derived category of a ring, we can afford a very explicit definition using derived functors. To link this with the more general theory of Grothendieck derivators, we refer the reader to \cite[\S5]{Sto} and the references therein. Let $I$ be a small category, and let $\ModR^I$ be the (Grothendieck) category of all $I$-shaped diagrams, that is functors from $I$ to $\ModR$. There is a natural equivalence between $\mathbf{C}(\ModR^I)$ and $\mathbf{C}(\ModR)^I$, and therefore we can consider $\Der(\ModR^I)$ as the derived category of $I$-shaped diagrams of complexes over $R$. Given such a diagram $X_i, i \in I$, its \EMP{homotopy colimit} is defined as $\hocolim_{i \in I} X_i := \mathbf{L}\colim_{i \in I} X_i$, where $\mathbf{L}\colim_{i \in I}: \Der(\ModR^I) \rightarrow \Der(R)$ is the left derived functor of the colimit functor.

Note that if $I$ is a \EMP{directed} small category, then $\colim_{i \in I} = \varinjlim_{i \in I}$ is an exact functor, and therefore $\hocolim_{i \in I} X_i \simeq \varinjlim_{i \in I} X_i$  (see \cite[Proposition 6.6]{Sto} for details). In this situation we talk about a \EMP{directed homotopy colimit}, and use the notation $\dhocolim_{i \in I} X_i$. Also, the exactness of the direct limit functors also yields $H^n(\dhocolim_{i \in I} X_i) \simeq \varinjlim_{i \in I}H^n(X_i)$ for any $n \in \mathbb{Z}$.

To see how the two constructions are related, consider $(X_n \mid n \geq 0) \in \Der(\ModR)^\omega$, a directed diagram of shape $\omega=(0 \rightarrow 1 \rightarrow 2 \rightarrow \cdots)$ inside the derived category. By \cite[Proposition 11.3 1)]{Kel}, there exists an object $(Y_n \mid n \geq 0)$ inside $\Der(\ModR^\omega)$, such that its natural interpretation in $\Der(\ModR)^\omega$ is the original diagram $(X_n \mid n \geq 0) \in \Der(\ModR)^\omega$ (the terminology here is that the \EMP{incoherent} diagram $(X_n \mid n \geq 0)$ lifts to a \EMP{coherent} diagram $(Y_n \mid n \geq 0)$). Then $\dhocolim_{n \geq 0}Y_n$ is quasi-isomorphic to $\Mcolim_{n \geq 0} X_n $, see \cite[Proposition 11.3 3)]{Kel}. In this way, Milnor colimits can be seen as particular cases of homotopy colimits, up to a (non-canonical) quasi-isomorphism.
\subsection{Homotopically smashing t-structures}
We call a t-structure $(\mathcal{U},\mathcal{V})$ \EMP{homotopically smashing} if the coaisle $\mathcal{V}$ is closed under taking \EMP{directed} homotopy colimits, that is, if for any directed diagram $(X_i \mid i \in I) \in \Der(\ModR^I)$ such that $X_i \in \mathcal{V}$ for all $i \in I$ we have $\dhocolim_{i \in I} X_i \in \mathcal{V}$. We remark some relations between t-structures and homotopy colimits:
\begin{prop}\label{P:properties}
	\begin{enumerate}
		\item[(i)] Any aisle of $\Der(R)$ is closed under arbitrary homotopy colimits, and any coaisle of $\Der(R)$ is closed under Milnor limits. 
		\item[(ii)] Any compactly generated t-structure is homotopically smashing.
	\end{enumerate}
\end{prop}
\begin{proof}
		\item[(i)]Any aisle is closed under homotopy colimits by \cite[Proposition 4.2]{SSV}. The closure of coaisles under Milnor limits is clear from them being closed under extensions, cosuspensions, and direct products. 

		\item[(ii)] This follows from \cite[Proposition 5.4]{SSV}, which implies that compact objects are precisely the objects $S$ with the property that 
		$$\Hom_{\Der(R)}(S,\dhocolim_{i \in I} X_i) \simeq \varinjlim_{i \in I} \Hom_{\Der(R)}(S,X_i)$$
		for any $(X_i \mid i \in I) \in \Der(\ModR^I)$.
\end{proof}

\subsection{Rigidity of aisles and coaisles}
	The following interplay between the rigid symmetric monoidal structure given by the derived tensor product on $\Der(R)$ and the t-structures will be essential in our endeavour.
	\begin{prop}\label{P:rigid}
		Let $(\mathcal{U},\mathcal{V})$ be a t-structure in $\Der(R)$. Let $X$ be a complex such that $X \in \Der^{\leq 0}$. Then the following holds:
			\begin{enumerate}
					\item[(i)] for any $U \in \mathcal{U}$ we have $X \otimes_R^\mathbf{L} U \in \mathcal{U}$,
					\item[(ii)] for any $V \in \mathcal{V}$ we have $\RHom_R(X,V) \in \mathcal{V}$.
			\end{enumerate}
	\end{prop}
	\begin{proof}
			We prove just $(ii)$, as $(i)$ follows by the same (in fact, simpler) argument (and is essentially proved in \cite[Corollary 5.2]{AJS2}).

			Fix a complex $V \in \mathcal{V}$. First note that, since $\mathcal{V}$ is closed under cosuspensions, $\RHom_R(R[i],V) \simeq V[-i] \in \mathcal{V}$ for any $i \geq 0$. Next, let $F \simeq R^{(\varkappa)}$ be a free $R$-module. Then $\RHom_R(F[i],V) \simeq \RHom_R(R[i],V)^\varkappa \in \mathcal{V}$ for any $i \geq 0$, as $\mathcal{V}$ is closed under direct products. Because $\mathcal{V}$ is extension-closed, it follows that $\RHom_R(X,V) \in \mathcal{V}$ whenever $X$ is a bounded complex of free $R$-modules concentrated in non-positive degrees. 

			Finally, let $X$ be any complex from $\Der^{\leq 0}$. By quasi-isomorphic replacement, we can without loss of generality assume that $X$ is a complex of free $R$-modules concentrated in non-positive degrees. Express $X$ as a Milnor colimit of its brutal truncations from below, $X = \Mcolim_{n \geq 0} \sigma^{>-n}X$ (see \ref{E:brutal}). Since $\sigma^{>-n}X$ is a bounded complex of free modules concentrated in non-positive degrees, $\RHom_R(\sigma^{>-n}X,V) \in \mathcal{V}$ by the previous paragraph. We compute, using (\ref{E:rhomlim}):
			$$\RHom_R(X,V) \simeq \RHom_R(\Mcolim_{n \geq 0}\sigma^{>-n}X,V) \simeq \Mlim_{n \geq 0} \RHom_R(\sigma^{>-n}X,V),$$
			and infer that $\RHom_R(X,V) \in \mathcal{V}$, as $\mathcal{V}$ is closed under Milnor limits.
	\end{proof}
	\subsection{Hereditary torsion pairs and the Thomason sets}
	We will also need to recall the notion of a torsion pair in the module category $\ModR$. A pair of subcategories $(\mathcal{T},\mathcal{F})$ of $\ModR$ is called a \EMP{torsion pair} if for any $M \in \ModR$:
	\begin{itemize}
		\item[(i)] $\Hom_R(\mathcal{T},M)=0$ if and only $M \in \mathcal{F}$, and
		\item[(ii)] $\Hom_R(M,\mathcal{F})=0$ if and only $M \in \mathcal{T}$.
	\end{itemize}
		A subcategory $\mathcal{T}$ of $\ModR$ belongs to some torsion pair $(\mathcal{T},\mathcal{F})$ if and only if it is closed under coproducts, extensions, and epimorphic images, and we call such a class a \EMP{torsion class}. Dually, we call a subcategory $\mathcal{F}$ closed under products, extensions, and submodules a \EMP{torsion-free class}. We say that a torsion pair $(\mathcal{T},\mathcal{F})$ is \EMP{hereditary}, if $\mathcal{T}$ is closed under submodules (or equivalently, $\mathcal{F}$ is closed under injective envelopes). As a shorthand, we say that $\mathcal{T}$ is a \EMP{hereditary torsion class} if it belongs as a torsion class into a hereditary torsion pair. In this situation, the torsion pair is fully determined by the cyclic modules in $\mathcal{T}$. A hereditary torsion pair $(\mathcal{T},\mathcal{F})$ is said to be \EMP{of finite type} if $\mathcal{F}$ is further closed under direct limits. This corresponds to the situation in which the torsion pair is determined by finitely presented cyclic objects in $\mathcal{T}$. This last statement can be made precise by tying such torsion pairs to certain data of geometrical flavour --- the Thomason subsets of the Zariski spectrum.

		Let $R$ be a commutative ring. A subset $X$ of the Zariski spectrum $\Spec(R)$ is called \EMP{Thomason (open)} if there is a family $\mathcal{I}$ of finitely generated ideals of $R$ such that $X = \bigcup_{I \in \mathcal{I}} V(I)$, where $V(I) = \{\pp \in \Spec(R) \mid I \subseteq \pp\}$ is the basic Zariski-closed set on $I$. The reason we used the parenthesized word ``open'' is that the Thomason sets are precisely the open subsets of the topological space $\Spec(R)$ endowed with the Hochster dual topology to the Zariski topology (for more details, see e.g. \cite[\S 2]{HS}). Their name comes from the Thomason's paper \cite{T}, in which he proved a bijective correspondence between Thomason sets and thick subcategories of $\Der^c(R)$. The following result appeared in \cite{GP}:
		\begin{thm}\label{T:thomasonpairs}\emph{(\cite[Theorem 2.2]{GP})}
			Let $R$ be a commutative ring. There is a 1-1 correspondence 
				$$\left \{ \begin{tabular}{ccc} \text{ Thomason subsets $X$} \\ \text{of $\Spec(R)$} \end{tabular}\right \}  \leftrightarrow  \left \{ \begin{tabular}{ccc} \text{Hereditary torsion pairs} \\ \text{$(\mathcal{T},\mathcal{F})$ of finite type in $\ModR$} \end{tabular}\right \},$$
				given by the assignment
				$$X \mapsto (\mathcal{T}_X,\mathcal{F}_X),$$
				where $\mathcal{T}_X = \{M \in \ModR \mid \Supp(M) \subseteq X\}$.
		\end{thm}
	\section{Bounded below compactly generated t-structures}
	The goal of this section is to establish a bijective correspondence between the compactly generated t-structures over an arbitrary commutative ring $R$ and infinite decreasing sequences of Thomason subsets of $\Spec(R)$, with the additional assumption that the t-structures are cohomologically bounded from below. As Theorem~\ref{T:thomasonpairs} suggests, an important role is played by hereditary torsion pairs --- we show that the aisle of any bounded below compactly generated t-structure is determined cohomology-wise by a decreasing sequence of hereditary torsion classes. We start with simple, but key observations about injective modules.
\begin{lem}\label{L:injdual}
	Let $P$ be a projective module over a commutative ring $R$.
\begin{enumerate}
	\item[(i)] If $E$ is an injective module, then $\Hom_R(P,E)$ is also an injective $R$-module.
	\item[(ii)] If $P$ is finitely generated, and $\iota_M: M \rightarrow E(M)$ is an injective envelope of $M$, then $\Hom_R(P,\iota_M)$ is an injective envelope of $\Hom_R(P,M)$.
\end{enumerate}
\begin{proof}
	\begin{enumerate}
		\item[(i)] The Hom-$\otimes$ adjunction yields a natural isomorphism
			$$\Hom_R(-,\Hom_R(P,E)) \simeq \Hom_R(- \otimes_R P, E).$$
			Since the latter functor is an exact functor $\ModR \rightarrow \ModR$, we conclude that $\Hom_R(P,E)$ is injective as an $R$-module.

	\item[(ii)] As $\Hom_R(R^{(n)},\iota_M) \simeq \iota_M^{(n)}: M^{(n)} \rightarrow E(M)^{(n)}$ as maps of $R$-modules, the result is clear from \cite[Proposition 6.16(2)]{AF} in the case when $P$ is free. The rest follows by considering $P$ as a direct summand in $R^{(n)}$ for some $n>0$.
	\end{enumerate}
\end{proof}
\end{lem}
\begin{lem}\label{L:stalkinj}
	Let $E$ be an injective $R$-module, and $X$ a complex. Then there is an isomorphism 
	$$\varphi: \Hom_{\Der(R)}(X,E[-n]) \simeq \Hom_R(H^n(X),E),$$
	for any $n \in \mathbb{Z}$, defined by the rule $\varphi(f) = H^n(f)$.
\end{lem}
\begin{proof}
	Since $\varphi$ is clearly an $R$-homomorphism, it is enough to show that $\varphi$ is a bijection. If $g \in \Hom_R(H^n(X),E)$, then we can lift it first to a map $g': Z^n(X) \rightarrow X$ such that $g'_{\restriction{B^n(X)}} = 0$, and then to a map $g'': X^n \rightarrow E$ by injectivity of $E$. The map $g''$ clearly extends to a map of complexes $\tilde{g}:X \rightarrow E[-n]$ such that $H^n(\tilde{g}) = g$, proving that $\varphi$ is surjective.

	Suppose that $f,g \in \Hom_{\Der(R)}(X,E[-n])$ are two maps inducing the same map on $n$-th cohomology. Since $E$ is injective, $\Hom_{\Der(R)}(X,E[-n]) \simeq \Hom_{\Htpy(R)}(X,E[-n])$. Put $h = f-g \in \Hom_{\Htpy(R)}(X,E[-n])$, and note that $h^n_{\restriction_{Z^n(X)}} = 0$. Then $h$ factors through the differential $d_X^n$. Therefore, by injectivity of $E$, there is a homotopy map $s: X^{n+1} \rightarrow E$ such that $h^n = s d_X^n$. It follows that $h$ is zero in $\Hom_{\Htpy(R)}(X,E[-n])$, and therefore $f=g$ in $\Hom_{\Der(R)}(X,E[-n])$, establishing that $\varphi$ is injective.
\end{proof}
The following result is crucial to this section, as it will allow us to replace bounded below complexes in compactly generated coaisles by collections of stalks of injective $R$-modules. Let $\projR$ denote the subcategory of all finitely generated projective $R$-modules.
\begin{lem}\label{L:injenvelope}
		Let $S \in \mathbf{K}^{-}(\projR)$ and $X \in \Der^{\geq k}$ for some $k \in \mathbb{Z}$. If $X \in S\Perp{\leq 0}$, then the stalk complex $E(H^k(X))[-k]$ belongs to $S\Perp{\leq 0}$.
\end{lem}
\begin{proof}
		Let $N$ be the largest integer such that the $N$-th coordinate of $S$ is non-zero, $S^N \neq 0$. Since $S$ is a bounded above complex of projectives, we may compute the homomorphisms starting in $S$ in the homotopy category, that is, $\Hom_{\Der(R)}(S,-) \simeq \Hom_{\Htpy(R)}(S,-)$. Therefore, if $N<k$, there is nothing to prove, so we may assume $N \geq k$. Further we proceed by backwards induction on $n=N,N-1,...,k$ and prove that $\Hom_{\Htpy(R)}(S[n-k],E(H^k(X))[-k])=0$. To save some ink, we fix the following notation:
		$$M:=H^k(X), E:=E(H^k(X)), H_n:=H^n(S), Z_n:=Z^n(S), B_n:=B^n(S), S_n := S^n$$
		$$(-)^E := \Hom_R(-,E), (-)^M := \Hom_R(-,M).$$
		By Lemma~\ref{L:stalkinj}, $\Hom_{\Htpy(R)}(S[n-k],E[-k]) \simeq \Hom_R(H_{n},E) = H_n^E$ for any $n$. Since $E$ is injective, the obviously obtained sequence
		\begin{equation}\label{EE11}
				0 \rightarrow B_{n+1}^E \rightarrow (S_{n}/B_{n})^E \rightarrow H_{n}^E \rightarrow 0
		\end{equation}
		is exact. 
		\begin{clm}
			The exact sequence (\ref{EE11}) is split.
		\end{clm}
		\begin{cprf}
		Consider the exact sequences
		$$0 \rightarrow B_{n+1+i}^E \rightarrow (S_{n+i}/B_{n+i})^E \rightarrow H_{n+i}^E \rightarrow 0$$
		$$0 \rightarrow (S_{n+i}/B_{n+i})^E \rightarrow S_{n+i}^E \rightarrow B_{n+i}^E \rightarrow 0$$
				for $i > 0$. By the induction hypothesis, $H_{n+i}^E =0$ for all $i>0$, and therefore $B_{n+1+i}^E \simeq (S_{n+i}/B_{n+i})^E$ for all $i>0$. Note that since $S_{N+1} = 0$, we have $B_{N+1}=0$. If $n=N$, this implies $B^E_{n+1} = 0$, and there is nothing to prove, so we may further assume that $N-n > 0$. Next we proceed by backwards induction on $i=N-n+1,N-n,\ldots,1$ to prove that $B_{n+i}$ is injective, and since the case $i=N-n+1$ is already done, we can assume $0<i \leq N-n$. Then the backwards induction hypothesis says that $B_{n+1+i}^E \simeq (S_{n+i}/B_{n+i})^E$ is injective, and therefore the second exact sequence above splits. This means that $B_{n+i}^E$ is a direct summand of $S_{n+i}^E$, which is injective by Lemma~\ref{L:injdual}, establishing the induction step. In particular, we showed that $B_{n+1}^E$ is injective, proving finally that (\ref{EE11}) is indeed split.
		\end{cprf}
		Now we are ready to prove that $H_n^E = 0$. Towards a contradiction, suppose that $H_n^E \neq 0$, and consider the following diagram, induced naturally by the Hom bifunctor:
		\begin{equation}\label{E:big}
		\begin{tikzcd}
				{} & S_{n-1}^M &  \arrow{l} Z_{n}^M  \arrow{dd}{h} & \arrow{l} H_n^M & \arrow{l} 0 \\
				S_{n-1}^M \arrow[equal]{ur} \arrow[dotted]{dd}[yshift=5ex]{\iota_{n-1}} &  \arrow{l} \arrow{ur}{g} \arrow[dotted]{dd}[yshift=5ex]{\iota_n} S_{n}^M \\
				{} & S_{n-1}^E  &  \arrow{l} Z_{n}^E & \arrow{l}[above]{l} H_n^E \arrow[shift left=2,hook]{dl}{i} & \arrow{l} 0  \\
				S_{n-1}^E \arrow[equal]{ur}  &  \arrow{l} \arrow{ur}{k} S_{n}^E & \arrow{l}{j} (S_n/B_n)^E \arrow[two heads, shift left=2]{ur}{\pi} & \arrow{l} 0 
		\end{tikzcd}
		\end{equation}
		By the left exactness and naturality of Hom-functors, all the rows of (\ref{E:big}) are exact and all the squares commute. Since $S_n$ and $S_{n-1}$ are finitely generated projective modules, and the ring $R$ is commutative, the front dotted vertical arrows are injective envelopes by Lemma~\ref{L:injdual}(ii). The map $\pi$ is a split epimorphism by the previous endeavour, with a section denoted in the diagram by $i$. Consider the submodule $(ji)[H_n^E]$ of $S_n^E$. By the essentiality of the vertical map $\iota_n$, the intersection $L=(ji)[H_n^E] \cap \iota_n[S_{n}^M]$ is non-zero. Choose an element $f \in S_n^M$ such that $\iota_n(f)$ is non-zero and belongs to $L$. Denote $\tilde{f}=g(f) \in Z_n^M$. We claim that $\tilde{f}$ is non-zero. Indeed, 
$$h(\tilde{f}) = hg(f) = k\iota_n(f),$$
		and because $\iota_n(f) \in L \subseteq (ji)[H_n^E]$, there is $f' \in H_n^E$ such that $\iota_n(f)=ji(f')$. Thus, we can compute further that
$$h(\tilde{f}) = k \iota_n(f) = kji(f') = l\pi i(f') = l(f') \neq 0,$$
using that $l$ is a monomorphism. Therefore, $\tilde{f}$ is non-zero as claimed.

		Note that by the commutativity of the squares of the diagram, $f$ is in the kernel of the map $S_n^M \rightarrow S_{n-1}^M$, since the vertical dotted map $\iota_{n-1}$ is injective. Therefore, $f d_S^{n-1} = 0$. Also, $\tilde{f}$ lies in the kernel of the map $Z_n^M \rightarrow S_{n-1}^M$, that is $\tilde{f}$ is a non-zero element of $H_n^M$. We proved that there is a map $f: S_n \rightarrow M$ such that it induces a non-zero map $\tilde{f}:H_n \rightarrow M$ on cohomology.

		Finally, we use this to infer that $\Hom_{\Htpy(R)}(S[n-k],X) \neq 0$, which is the desired contradiction. Indeed let us define a map $\varphi: S[n-k] \rightarrow X$ by first letting $\varphi^k: S_n \rightarrow Z^k(X) \subseteq X^k$ be some map lifting the map $f: S_n \rightarrow M = Z^k(X)/B^k(X)$, using projectivity of $S_n$. The other coordinates of $\varphi$ are defined as follows: put $\varphi_j=0$ for all $j>k$, and define $\varphi_i$ for $i<k$ inductively by projectivity of coordinates of $S$ --- using the exactness of $X$ in degrees smaller than $k$, and the fact that the image of the composition $\varphi_k d_S^{n-1}$ is contained in $B^k(X)$, which in turn follows from $f d_S^{n-1} = 0$.
		
		Therefore, we obtained a map of complexes $\varphi: S[n-k] \rightarrow X$ which is not null-homotopic, because $H^k(\varphi)=\tilde{f}$ is non-zero. The only way around this contradiction is that $0=H_n^E=\Hom_R(H_n,E)\simeq \Hom_{\Htpy(R)}(S[n-k],E[-k])$, establishing the induction step.
\end{proof}
			We say that a t-structure $(\mathcal{U},\mathcal{V})$ is \EMP{bounded below} if $\mathcal{V} \subseteq \Der^+(R)$. Since $\mathcal{V}$ is closed under products, this is equivalent to $\mathcal{V} \subseteq \Der^{\geq m}$ for some integer $m$. In the following lemma, we show that the condition from Lemma~\ref{L:injenvelope} can be used to replace a bounded below complex in a coaisle of a compactly generated t-structure by a well-chosen injective resolution --- one, such that each of its component considered as a stalk complex in the appropriate degree belongs to the coaisle. 


\begin{lem}\label{L:injcogen}
		Let $(\mathcal{U},\mathcal{V})$ be a bounded below t-structure. Suppose that the following implication holds:
		$$X \in \mathcal{V} \Rightarrow E(H^{\inf(X)}(X))[-\inf(X)] \in \mathcal{V}.$$
		Then there is a collection $\mathcal{E}$ of (shifts of) stalk complexes of injective modules such that $(\mathcal{U},\mathcal{V})$ is cogenerated by $\mathcal{E}$.
\end{lem}
\begin{proof}
	Fix a complex $X = X_0 \in \mathcal{V}$ and let $k=\inf(X_0)$. Since $\mathcal{V}$ is bounded below, $k \in \mathbb{Z}$ and we can denote $E_0=E(H^k(X_0)) \in \ModR$. By the assumption, the stalk complex $E_0[-k]$ belongs to $\mathcal{V}$. Also, by Lemma~\ref{L:stalkinj}, there is a map of complexes $\iota_0: X_0 \rightarrow E_0[-k]$ inducing the injective envelope $H^k(\iota_0): H^k(X_0) \rightarrow E_0$. Consider the induced triangle:
		$$X_0 \xrightarrow{\iota_0} E_0[-k] \rightarrow \Cone(\iota_0) \rightarrow X_0[1].$$
		Since $\mathcal{V}$ is closed for extensions and cosuspensions, we obtain that 
		$$X_1:=\Cone(\iota_0)[-1] \in \mathcal{V}.$$ 
		By the mapping cone construction, $X_1$ can be represented by the following complex:
		$$\cdots \xrightarrow{d_X^{k-1}} X^k \xrightarrow{\begin{pmatrix}d_X^k \\ \iota_0^k\end{pmatrix}} X^{k+1} \oplus E_0 \xrightarrow{\begin{pmatrix}d_X^{k+1} & 0\end{pmatrix}} X^{k+2} \xrightarrow{d_X^{k+2}} \cdots$$
		From this, one can see that $H^k(X_1)=0$, and in fact $\inf(X_1)>k$. Continuing inductively we obtain a sequence of complexes $X_n \in \mathcal{V}$ and injective modules $E_n$ for each $n \geq 0$ such that $\inf(X_n) \geq k+n$, and $E_n[-k-n] \in \mathcal{V}$, together with triangles
		$$X_{n+1} \xrightarrow{f_n} X_n \xrightarrow{\iota_n} E_n[-k-n] \rightarrow X_{n+1}[1]$$
		for all $n \geq 0$, where $H^{k+n}(\iota_n): H^{k+n}(X_n) \rightarrow E_n$ is the injective envelope map. Note that $f_n: X_{n+1} \rightarrow X_n$ can be, up to quasi-isomorphism, represented by a degree-wise surjective map of complexes
		\begin{equation}\label{E:repre}
		\begin{tikzcd}[ampersand replacement=\&]
			 \cdots \arrow{r}{} \& X^{k+n} \oplus E_{n-1} \arrow[equal]{d}\arrow{r}{\begin{pmatrix}d^{k+n}_X & 0 \\ \alpha_{n} &  \epsilon_{n} \end{pmatrix}} \& X^{k+n+1} \oplus E_{n} \arrow{d}{\pi_{k+n+1}}\arrow{r} \& X^{k+n+2} \arrow[equal]{d}\arrow{r} \& \cdots\\
			 \cdots \arrow{r} \& X^{k+n} \oplus E_{n-1} \arrow{r} \& X^{k+n+1} \arrow{r} \& X^{k+n+2} \arrow{r} \& \cdots\\
		\end{tikzcd}
		\end{equation}
		where the vertical map $\pi_{k+n+1}$ is the obvious split epimorphism, and the maps $\alpha_{n}$, $\epsilon_{n}$ are the coordinates of the map $\iota^{k+n}_{n}$,
		$$\iota^{k+n}_{n}: X^{k+n}_{n} = X^{k+n} \oplus E_{n-1} \xrightarrow{\begin{pmatrix}\alpha_{n} & \epsilon_{n}\end{pmatrix}} E_n,$$ 
			where $\iota_n$ is the map obtained from the inductive construction above. (We let $E_{i} = 0$ for $i<0$ and $\alpha_j=\epsilon_j=0$ for $j\leq 0$ so that (\ref{E:repre}) above makes sense for any $n \geq 0$.) We define a complex $Z = \varprojlim_{n \geq 0}X_n$, where the maps in the inverse system are as in (\ref{E:repre}). Then $Z$ is a complex of form
		$$\cdots X^{k+n} \oplus E_{n-1} \xrightarrow{\begin{pmatrix}d^{k+n}_X & 0 \\ \alpha_{n} &  \epsilon_{n} \end{pmatrix}} X^{k+n+1} \oplus E_n \xrightarrow{\begin{pmatrix}d^{k+n+1}_X & 0 \\ \alpha_{n+1} &  \epsilon_{n+1} \end{pmatrix}} X^{k+n+2} \oplus E_{n+1} \cdots$$
			There is a triangle in $\Der(R)$
		$$E \rightarrow Z \xrightarrow{f} X \rightarrow E[1],$$
		where $f: Z \rightarrow X = X_0$ is the limit map. From the construction, we see that those coordinates of $f$, which are not just identities, consist of the split epimorphisms $f^i = \pi_{i}, i>k$. It follows that $E = \Cone(f)[-1] \simeq \Ker(f)$ is a complex of form
		$$\cdots \rightarrow 0 \rightarrow E_0 \xrightarrow{\epsilon_1} E_1 \xrightarrow{\epsilon_2} E_2 \xrightarrow{\epsilon_3} \cdots,$$
		where the coordinate $E_0$ is in degree $k+1$. Since $\inf(X_n) \geq k+n$ for any $k \geq 0$, we infer that $Z$ is exact, and thus $X$ is quasi-isomorphic to $E[1]$. Let $\mathcal{Y}_{X} = (\Perp{\leq 0}\{E_n[-n-k] \mid n \geq 0\})\Perp{0}$ be a subcategory of $\Der(R)$. Because $\mathcal{Y}_{X}$ is closed under extensions,  cosuspensions, and products, and thus also under Milnor limits, we have by (\ref{E:brutal}) that $E[1] \in \mathcal{Y}_{X}$, and thus $X \in \mathcal{Y}_{X}$.

		Repeating this for all complexes $X \in \mathcal{V}$, we obtain classes $\mathcal{Y}_X = (\Perp{\leq 0}\mathcal{E}_X)\Perp{0}$, where $\mathcal{E}_X$ is a collection of stalks of injective modules with $\mathcal{E}_X \subseteq \mathcal{V}$, such that $X \in \mathcal{Y}_X$. Then also $\mathcal{Y}_X \subseteq \mathcal{V}$ for each $X \in \mathcal{V}$. Put $\mathcal{E} = \bigcup_{X \in \mathcal{V}}\mathcal{E}_X$. Then we have $\mathcal{E} \subseteq \mathcal{V}$, and $X \in \mathcal{Y}_X \subseteq (\Perp{\leq 0}\mathcal{E})\Perp{0}$ for each $X \in \mathcal{V}$, and thus $\mathcal{V} = (\Perp{\leq 0}\mathcal{E})\Perp{0}$. Because $\mathcal{V}$ is closed under cosuspensions, so is $\mathcal{E}$, and therefore $\mathcal{V} = (\Perp{0}\mathcal{E})\Perp{0}$. In other words, the t-structure $(\mathcal{U},\mathcal{V})$ is cogenerated by $\mathcal{E}$.
		\end{proof}

	\begin{prop}\label{P:herpairs}
		For a bounded below t-structure $(\mathcal{U},\mathcal{V})$, the following conditions are equivalent:
			\begin{enumerate}
					\item[(i)] for any $X \in \mathcal{V}$, $E(H^{\inf X}(X))[-\inf(X)] \in \mathcal{V}$,
					\item[(ii)] $(\mathcal{U},\mathcal{V})$ is cogenerated by stalks of injective modules,
					\item[(iii)] there is a decreasing sequence
							$$\cdots \supseteq \mathcal{T}_n \supseteq \mathcal{T}_{n+1} \supseteq \mathcal{T}_{n+2} \supseteq \cdots$$
							of hereditary torsion classes such that $\mathcal{U}=\{X \in \Der(R) \mid H^n(X) \in \mathcal{T}_n\}$.
			\end{enumerate}
	\end{prop}
	\begin{proof}
		$(i) \Rightarrow (ii)$: Lemma~\ref{L:injcogen}.

			$(ii) \Rightarrow (iii)$: For each $n \in \mathbb{Z}$, let $\mathcal{E}_n$ be a collection of injective modules such that 
			$$\mathcal{U}=\{X \in \Der(R) \mid \Hom_{\Der(R)}(X,E[-n])=0 ~\forall E \in \mathcal{E}_n ~\forall n \in \mathbb{Z}\}.$$ 
			By injectivity and Lemma~\ref{L:stalkinj} we can rewrite the class as:
			$$\mathcal{U}=\{X \in \Der(R) \mid \Hom_{R}(H^n(X),E)=0 ~\forall E \in \mathcal{E}_n ~\forall n \in \mathbb{Z}\}.$$
			Therefore, 
			$$\mathcal{U}=\{X \in \Der(R) \mid H^n(X) \in \mathcal{T}_n\},$$
			where $\mathcal{T}_n$ is the torsion class cogenerated by $\mathcal{E}_n$. Again by injectivity, it follows that $\mathcal{T}_n$ is hereditary for each $n \in \mathbb{Z}$. Since $\mathcal{U}$ is closed under suspensions, necessarily $\mathcal{T}_n \supseteq \mathcal{T}_{n+1}$ for each $n \in \mathbb{Z}$.

			$(iii) \Rightarrow (i)$: Towards contradiction, suppose that there is $X \in \mathcal{V}$ such that $E = E(H^n(X))[-n] \not\in \mathcal{V}$, where $n=\inf(X)$. The injectivity of $E$ together with Lemma~\ref{L:stalkinj} implies that $E$ does not belong to $\mathcal{F}_n$, the torsion-free class of the hereditary torsion pair $(\mathcal{T}_n,\mathcal{F}_n)$. Therefore, $H^n(X) \not\in \mathcal{F}_n$, and whence there is a non-zero map $f: T \rightarrow H^n(X)$, for some $T \in \mathcal{T}_n$. Since $n = \inf(X)$, we can extend $f$ to a map of complexes $\tilde{f}: T[-n] \rightarrow X$ such that $H^n(\tilde{f})=f$. Since $T[-n] \in \mathcal{U}$, this is a contradiction with $X \in \mathcal{V}$.
	\end{proof}
	The next auxiliary lemma says that, even though the injective envelope is not a functorial construction, we can still apply it to well-ordered directed systems in a natural way.
	\begin{lem}\label{L:injdirected}
		Let $(M_\alpha, f_{\alpha,\beta} \mid \alpha < \beta < \lambda)$ be a well-ordered directed system in $\ModR$. Then there is a directed system $(E_\alpha, g_{\alpha,\beta} \mid \alpha < \beta < \lambda)$ such that $E_\alpha = E(M_\alpha)$, and such that the natural embeddings $M_\alpha \subseteq E_\alpha$ induce a homomorphism between the two directed systems.
	\end{lem}
\begin{proof}
	We construct maps $g_{\alpha,\beta}$ by the induction on $\beta<\lambda$. Suppose that we have already constructed maps $g_{\alpha,\beta}$ for all $\alpha < \beta < \gamma$ for some $\gamma < \lambda$, so that $(E_\alpha, g_{\alpha,\beta} \mid \alpha < \beta < \gamma)$ forms a directed system with the claimed properties. If $\gamma$ is a successor, we first let $g_{\gamma-1,\gamma}: E_{\gamma-1} \rightarrow E_\gamma$ be any map extending $f_{\gamma-1,\gamma}$, which exists by injectivity. For any $\alpha<\gamma$ we then put $g_{\alpha,\gamma} = g_{\gamma-1,\gamma} g_{\alpha,\gamma-1}$. This is easily seen to define a directed system, and $g_{\alpha,\gamma}{}_{\restriction{M_\alpha}} = g_{\gamma-1,\gamma}  g_{\alpha,\gamma-1}{}_{\restriction{M_\alpha}} = g_{\gamma-1,\gamma}  f_{\alpha,\gamma-1} = f_{\gamma-1,\gamma}  f_{\alpha,\gamma-1} = f_{\alpha,\gamma}$.

				Suppose now that $\gamma$ is a limit ordinal. By the inductive assumption, the system $(E_\alpha, g_{\alpha,\beta} \mid \alpha < \beta < \gamma)$ is a direct system, and we let $L = \varinjlim_{\alpha < \gamma} E_\alpha$. Denote by $h_\alpha : E_\alpha \rightarrow L$ the limit maps. By the exactness of direct limit, $\varinjlim_{\alpha < \gamma} M_\alpha$ is naturally a subobject in $L$. Then there is a map $l: L \rightarrow E_\gamma$ extending the universal map $\varinjlim_{\alpha < \gamma} M_\alpha \rightarrow M_\gamma$. We put $g_{\alpha,\gamma} = l  h_\alpha$ for any $\alpha<\gamma$. For any $\alpha < \beta < \gamma$, we have $g_{\beta,\gamma}  g_{\alpha,\beta}= l  h_\beta  g_{\alpha, \beta} = l  h_\alpha = g_{\alpha,\gamma}$. Therefore, $(E_\alpha, g_{\alpha,\beta} \mid \alpha<\beta<\gamma+1)$ forms a directed system. By the construction, $g_{\alpha,\beta}$ equals $f_{\alpha,\beta}$ after restriction to $M_\alpha$ for any $\alpha < \beta < \gamma+1$. This establishes the induction step, and therefore also the proof.
\end{proof}
\begin{lem}\label{L:finitetype}
		Let $(\mathcal{U},\mathcal{V})$ be a bounded below t-structure satisfying the conditions of Proposition~\ref{P:herpairs}. If $(\mathcal{U},\mathcal{V})$ is homotopically smashing, then the hereditary torsion pairs $(\mathcal{T}_n,\mathcal{F}_n)$ of Proposition~\ref{P:herpairs} are of finite type.
\end{lem}
\begin{proof}
		We proceed by contradiction. Suppose that there is $n \in \mathbb{Z}$ such that the torsion-free class $\mathcal{F}_n$ is not closed under direct limits. Then by \cite[Lemma 2.14]{GT}, there is a well-ordered directed system $(M_\alpha \mid \alpha < \lambda)$ such that $M_\alpha \in \mathcal{F}_n$ for each $\alpha < \lambda$, but $M = \varinjlim_{\alpha < \lambda} M_i \not\in \mathcal{F}_n$. Let $(E_\alpha \mid  \alpha < \lambda)$ be the directed system provided by Lemma~\ref{L:injdirected}. Since $\mathcal{F}_n$ is closed under injective envelopes, and by the Lemma~\ref{L:injdirected} $M$ embeds into $L=\varinjlim_{\alpha < \lambda}E_\alpha$, this is a directed system of injective modules from $\mathcal{F}_n$, such that its direct limit does not belong to $\mathcal{F}_n$.

		Because $E_\alpha \in \mathcal{F}_n$ is injective, it is easy using Lemma~\ref{L:stalkinj} to check that $E_\alpha[-n] \in \mathcal{V}$ for all $\alpha<\lambda$. Since the direct limit $L$ is not in $\mathcal{F}_n$, there is a non-zero map from some module from $\mathcal{T}_n$ to $L$, and thus $L[-n] \not\in \mathcal{V}$. Therefore, $L[-n] = \dhocolim_{\alpha < \lambda} E_\alpha[-n] \not\in \mathcal{V}$, which is in contradiction with $(\mathcal{U},\mathcal{V})$ being homotopically smashing.
\end{proof}
	\begin{rem}
	We sketch here an alternative proof of Lemma~\ref{L:finitetype}, which relies on a deep result from \cite{SSV}, but is perhaps more conceptual. It is relatively straightforward to show that the torsion radical $t_n$ associated to the torsion pair $(\mathcal{T}_n,\mathcal{F}_n)$ can be expressed as $t_n(M) = H^n(\tau_{\mathcal{U}}(M[-n]))$ for all $M \in \ModR$. One can show using the result \cite[Theorem 3.1]{SSV} that $\tau_\mathcal{U}$ naturally preserves directed homotopy colimits. This implies that $t_n(M)$ commutes with direct limits, which in turn implies that $(\mathcal{T}_n,\mathcal{F}_n)$ is of finite type.
\end{rem}
A \EMP{Thomason filtration} of $\Spec(R)$ is a decreasing map $\Phi: \mathbb{Z} \rightarrow (2^{\Spec(R)},\subseteq)$ such that $\Phi(i)$ is a Thomason subset of $\Spec(R)$ for each $i \in \mathbb{Z}$. To any Thomason filtration, assign the following class:

$$\mathcal{U}_\Phi = \{X \in \Der(R) \mid \Supp H^n(X) \subseteq \Phi(n) ~\forall n \in \mathbb{Z}\}.$$

Given $x \in R$, the \EMP{Koszul complex of $x$} is the complex 
	$$K(x) = \cdots 0 \rightarrow R \xrightarrow{\cdot x} R \rightarrow 0 \rightarrow \cdots$$
concentrated in degrees $-1$ and $0$. If $\bar{x}=(x_1,x_2,\ldots,x_n)$ is a sequence of elements of $R$, we define the \EMP{Koszul complex of $\bar{x}$} by tensor product: $K(x_1,x_2,\ldots,x_n) = \bigotimes_{i=1}^n K(x_i)$.

\begin{conv}
		\begin{enumerate}
				\item[(i)] Let $I$ be a finitely generated ideal, and let $\bar{x}$ and $\bar{y}$ be two finite sequences of generators of $I$. It is \EMP{not} true in general that $K(\bar{x})$ and $K(\bar{y})$ are quasi-isomorphic --- see \cite[Proposition 1.6.21]{BH}. Nevertheless, we will for each finitely generated ideal $I$ fix once and for all a finite sequence $\bar{x}_I$ of generators, and let $K(I):=K(\bar{x}_I)$. Our results will \EMP{not} depend on the choice of the generating sequence. The reason behind this is that although the quasi-isomorphism class of the Koszul complex does depend on the choice of generators, the vanishing of the relative cohomology does not --- see \cite[Corollary 1.6.22 and Corollary 1.6.10(d)]{BH}.
				\item[(ii)] In the following, we will often need to enumerate over all finitely generated ideals $I$ such that the basic Zariski-closed set $V(I) = \{\pp \in \Spec(R) \mid I \subseteq \pp\}$ is a subset of some Thomason set $X$. For brevity, we will use a shorthand quantifier ``$\forall V(I) \subseteq X$'' instead. There is no risk of confusion, because even though the shorthand would make sense for infinitely generated ideals as well, these would either lead to an undefined expressions $K(I)$ or $\check{C}^{\sim}(I)$ with $I$ infinitely generated, or to inclusion of cyclic modules $R/I$ with $I$ infinitely generated, which would not endanger the validity of our results.
		\end{enumerate}
\end{conv}

We can gather our findings about bounded below compactly generated t-structure in this way:
\begin{lem}\label{L:boundedbelow}
	Let $R$ be a commutative ring and $(\mathcal{U},\mathcal{V})$ a bounded below t-structure. Then the following are equivalent:
		\begin{itemize}
			\item[(i)] $(\mathcal{U},\mathcal{V})$ is compactly generated,
			\item[(ii)] $(\mathcal{U},\mathcal{V})$ is homotopically smashing, and 				
					$$X \in \mathcal{V} \Rightarrow E(H^{\inf(X)}(X))[-\inf(X)] \in \mathcal{V},$$
			\item[(iii)] there is a Thomason filtration $\Phi$ such that there is $k \in \mathbb{Z}$ with $\Phi(k) = \Spec(R)$, and $\mathcal{U} = \mathcal{U}_\Phi$,
			\item[(iv)] there is a Thomason filtration $\Phi$ such that there is $k \in \mathbb{Z}$ with $\Phi(k) = \Spec(R)$, and such that the t-structure $(\mathcal{U},\mathcal{V})$ is generated by the set
					$$\mathcal{S}_\Phi = \{K(I)[-n] \mid ~\forall V(I) \subseteq \Phi(n) ~\forall n \in \mathbb{Z}\}.$$
		\end{itemize}
\end{lem}
\begin{proof}
		$(i) \Rightarrow (ii):$ First, by Proposition~\ref{P:properties} any compactly generated t-structure is homotopically smashing. Let $\mathcal{S} \subseteq \Der^c(R)$ be such that $\mathcal{V} = \mathcal{S}\Perp{0}$. Because $\mathcal{V}$ is closed under cosuspensions, we also have $\mathcal{V} = \mathcal{S}\Perp{\leq 0}$. By the assumption, there is $k \in \mathbb{Z}$ such that $\mathcal{V} \subseteq \Der^{\geq k}$. Therefore, we can apply Lemma~\ref{L:injenvelope} to infer that for any $X \in \mathcal{V}$, $E(H^{\inf(X)}(X))[-\inf(X)] \in S\Perp{\leq 0}$ for all $S \in \mathcal{S}$, and therefore $E(H^{\inf(X)}(X))[-\inf(X)] \in \mathcal{V}$.
	
		$(ii) \Rightarrow (iii):$ Using Proposition~\ref{P:herpairs} and Lemma~\ref{L:finitetype}, there is a sequence $\mathcal{T}_n$ of hereditary torsion pairs $(\mathcal{T}_n,\mathcal{F}_n)$ of finite type such that $\mathcal{U}=\{X \in \Der(R) \mid H^n(X) \in \mathcal{T}_n ~\forall n \in \mathbb{Z}\}$. For each $n \in \mathbb{Z}$, there is a Thomason set $\Phi(n)$ such that $\mathcal{T}_n = \{M \in \ModR \mid \Supp(M) \subseteq \Phi(n)\}$, using Theorem~\ref{T:thomasonpairs}. Because $\mathcal{U}$ is closed under suspensions, the map $\Phi: \mathbb{Z} \rightarrow 2^{\Spec(R)}$ is a Thomason filtration. We conclude that $\mathcal{U} = \mathcal{U}_\Phi$. Finally, since $(\mathcal{U},\mathcal{V})$ is bounded below, there is necessarily an integer $k$ such that $\Phi(k) = \Spec(R)$.
		
		$(iii) \Rightarrow (iv):$ Let $(\mathcal{U}',\mathcal{V}')$ be the t-structure generated by the set $\mathcal{S}_\Phi$. By \cite[Proposition 1.6.5b]{BH}, $\Supp H^k(K(I)) \subseteq V(I)$ for any $k \in \mathbb{Z}$, and thus $K(I)[-n] \in \mathcal{U}_\Phi$ for all $n \in \mathbb{Z}$ and any finitely generated ideal $I$ such that $V(I) \subseteq \Phi(n)$. Therefore, $\mathcal{U}' \subseteq \mathcal{U} = \mathcal{U}_\Phi$.

		For the converse inclusion, note that $\mathcal{S}_\Phi \subseteq \Der^c(R)$. Then we can use the already proven implication $(i) \Rightarrow (iii)$ to infer that there is a Thomason filtration $\Phi'$ such that $\mathcal{U}' = \mathcal{U}_{\Phi'}$. Since $\mathcal{U}' \subseteq \mathcal{U}_\Phi$, we have $\Phi'(n) \subseteq \Phi(n)$ for all $n \in \mathbb{Z}$. But since by \cite[p. 360, 8.2.7]{NC}, $H^0(K(I)) \simeq R/I$, we have $\Supp(H^n(K(I)[-n])) = V(I)$. Therefore, necessarily $\Phi=\Phi'$, and thus $\mathcal{U}' = \mathcal{U}_{\Phi'} = \mathcal{U}_\Phi = \mathcal{U}$.

		$(iv) \Rightarrow (i):$ Clear, since Koszul complexes are compact objects of $\Der(R)$.
\end{proof}
	We call a Thomason filtration $\Phi$ \EMP{bounded below} if there is an integer $k$ such that $\Phi(k) = \Spec(R)$. Then we can formulate Lemma~\ref{L:boundedbelow} in a form of a bijective correspondence:
\begin{thm}\label{T:bounded}
	Let $R$ be a commutative ring. Then there is a 1-1 correspondence:
		$$\left \{ \begin{tabular}{ccc} \text{ Bounded below Thomason} \\ \text{filtrations $\Phi$ of $\Spec(R)$}  \end{tabular}\right \}  \leftrightarrow  \left \{ \begin{tabular}{ccc} \text{Bounded below compactly generated } \\ \text{t-structures in $\Der(R)$} \end{tabular}\right \}.$$
		The correspondence is given by mutually inverse assignments
		$$\Phi \mapsto (\mathcal{U}_\Phi,\mathcal{U}_\Phi\Perp{0}),$$
		and
		$$(\mathcal{U},\mathcal{V}) \mapsto \Phi_{\mathcal{U}},$$
		where
		$$\Phi_{\mathcal{U}}(n) = \bigcup_{M \in \ModR, M[-n] \in \mathcal{U}} \Supp(M).$$
\end{thm}

\section{A version of the telescope conjecture for bounded below t-structures over commutative noetherian rings}
	The approach of the previous section can be used to extract more information about bounded below t-structures in the case of a commutative noetherian ring --- for such t-structures, the homotopically smashing property is equivalent to compact generation. This can be seen as a ``semistable'' version of the telescope conjecture for bounded below t-structures in place of localizing pairs. As an application, we will use this in Section~\ref{S:cosilting} to establish a cofinite type result for cosilting complexes over commutative noetherian rings.
	\begin{lem}\label{L:loc}
		Let $R$ be a commutative ring and $(\mathcal{U},\mathcal{V})$ a homotopically smashing t-structure in $\Der(R)$. Let $\pp$ be a prime and set
		$$\mathcal{U}_{\pp}=\{X \otimes_R R_{\pp} \mid X \in \mathcal{U}\},$$
		$$\mathcal{V}_{\pp}=\{X \otimes_R R_{\pp} \mid X \in \mathcal{V}\}.$$
		Then the pair $(\mathcal{U}_{\pp},\mathcal{V}_{\pp})$ is a homotopically smashing t-structure in $\Der(R_{\pp})$.
	\end{lem}
	\begin{proof}
		Because $R_{\pp}$ is a flat $R$-module, it can be by \cite[Th\'{e}or\`{e}me 1.2]{Laz} written as a direct limit of finite free modules, $R_{\pp} = \varinjlim_{i \in I}R^{n_I}$. Since $\mathcal{V}$ is closed under homotopy colimits, then 
		$$X_{\pp} := X \otimes_R R_{\pp} = X \otimes_R \varinjlim_{i \in I}R^{n_I} \simeq  \varinjlim_{i \in I}X^{n_I} \simeq \dhocolim_{i \in I}X^{n_I}  \in \mathcal{V}$$ 
		for any $X \in \mathcal{V}$ and prime $\pp$, and thus $\mathcal{V}_{\pp} \subseteq \mathcal{V}$. Similarly, $\mathcal{U}_{\pp} \subseteq \mathcal{U}$, and $\mathcal{U}_{\pp}$ is clearly closed under suspensions. It remains to show that $(\mathcal{U}_{\pp},\mathcal{V}_{\pp})$ provides approximation triangles in $\Der(R_{\pp})$. Let $Y \in \Der(R_{\pp})$, then there is an approximation triangle with respect to $(\mathcal{U},\mathcal{V})$ in $\Der(R)$:
		\begin{equation}\label{E:approxloc}
			\tau_{\mathcal{U}} Y \rightarrow Y \rightarrow \tau_{\mathcal{V}}Y \xrightarrow{} \tau_{\mathcal{U}} Y[1].
		\end{equation}
			Localizing this triangle at $\pp$ yields the desired approximation triangle with respect to $(\mathcal{U}_{\pp},\mathcal{V}_{\pp})$ (and in fact, by uniqueness of approximation triangles, the triangle (\ref{E:approxloc}) is already in $\Der(R_{\pp})$).  Finally, $\mathcal{V}_{\pp}$ is clearly closed under directed homotopy colimits, as a directed homotopy colimit commutes with localization.
	\end{proof}
	As in the Neeman's proof of the telescope conjecture for localizing pairs in \cite{N}, we will use the Matlis' theory of injectives in a crucial way. For this reason, we state explicitly the structural result for injective modules in the setting of a commutative noetherian ring:
	\begin{thm}\label{T:Matlis}\emph{(\cite{M})}
		Let $R$ be a commutative noetherian ring. Then any injective module is isomorphic to a direct sum of indecomposable injective modules of form $E(R/\pp)$, where $\pp \in \Spec(R)$. Furthermore, each of the modules $E(R/\pp)$ admits a filtration by $\kappa(\pp)$-modules.
	\end{thm}
\begin{thm}\label{T:TC}
		Let $R$ be a commutative noetherian ring. If $(\mathcal{U},\mathcal{V})$ is a bounded below homotopically smashing t-structure in $\Der(R)$, then it is compactly generated.
	\end{thm}

	\begin{proof}
			Since $(\mathcal{U},\mathcal{V})$ is bounded below, it is by Lemma~\ref{L:boundedbelow} enough to show that for any complex $X \in \mathcal{V}$, we have 
			\begin{equation}\label{E92}
			E(H^l(X))[-l] \in \mathcal{V},
			\end{equation}
			where $l=\inf(X)$. Since $\mathcal{V}$ is closed under directed homotopy colimits, and thus in particular under direct sums and direct limits of modules inhabiting some cohomological degree of $\mathcal{V}$ as stalk complexes, it is enough to show that $\kappa(\pp)[-l] \in \mathcal{V}$, whenever $\pp \in \Ass(H^l(X))$. Indeed, by Theorem~\ref{T:Matlis} we know that $E(H^l(X))$ is isomorphic to a direct sum of indecomposable injectives of form $E(R/\pp)$, where $\pp \in \Ass(H^l(X))$, and each of these is filtered by $\kappa(\pp)$-modules. Since any $\kappa(\pp)$-module is isomorphic to a direct sum of copies of $\kappa(\pp)$, we conclude that $\kappa(\pp)[-l] \in \mathcal{V}$ for each $\pp \in \Ass(H^l(X))$ implies (\ref{E92}).

			Fix $\pp \in \Ass(H^l(X))$. By Lemma~\ref{L:loc}, $(\mathcal{U}_{\pp},\mathcal{V}_{\pp})$ is a homotopically smashing t-structure, and obviously $X_{\pp} := X \otimes_R R_{\pp} \in \mathcal{V}_p \subseteq \mathcal{V}$. We show that $\kappa(\pp)[-l] \in \mathcal{V}_{\pp}$. By Proposition~\ref{P:rigid}, the complex $Y=\RHom_R(\kappa(\pp),X_{\pp}) \in \mathcal{V}_{\pp}$. Since $R/\pp$ embeds into $H^l(X)$, the residue field $\kappa(\pp)$ embeds into $H^l(X_{\pp}) \simeq H^l(X)_{\pp}$. Since $l=\inf(X_{\pp})$, there is a map of complexes $\kappa(\pp)[-l] \rightarrow X_{\pp}$ inducing the embedding $\kappa(\pp) \subseteq H^l(X_{\pp})$ in the $l$-th cohomology. Using \cite[15.68.0.2]{Stacks}, it follows that 
			$$0 \neq \Hom_{\Der(R_{\pp})}(\kappa(\pp)[-l],X_{\pp}) \simeq \Hom_{\Der(R_{\pp})}(\kappa(\pp),X_{\pp}[l]) \simeq$$ 
			$$\simeq H^l \RHom_{R_{\pp}}(\kappa(\pp),X_{\pp}) = H^l(Y).$$
			Because $Y \in \Der(\kappa(\pp)) \subseteq \Der(R_{\pp})$, $Y$ is quasi-isomorphic in $\Der(R_{\pp})$ to a complex of $\kappa(\pp)$-modules. As $\kappa(\pp)$ is a field, $Y$ is --- up to quasi-isomorphism --- a split complex, and thus $H^l(Y)[-l]$ is a direct summand of $Y$. Therefore, $H^l(Y)[-l] \in \mathcal{V}_{\pp}$. As $H^l(Y)[-l] \neq 0$ is a vector space over $\kappa(\pp)$, we conclude that $\kappa(\pp)[-l] \in \mathcal{V}_{\pp} \subseteq \mathcal{V}$, as desired.
	\end{proof}

\section{Compactly generated t-structures in $\Der(R)$}
	The purpose of this section is to extend Theorem~\ref{T:bounded} to all compactly generated t-structures in $\Der(R)$, that is, without any assumption of a cohomological bound. To do this, we need to adopt a different approach, which we shortly explain now. Following \cite{SP}, compactly generated t-structures (and by duality, also compactly generated co-t-structures) correspond to full subcategories of $\Der^c(R)$ closed under extensions, direct summands, and suspensions. This indicates that, in order to classify compactly generated t-structures, it is enough to consider the generation of compact objects in the aisle. In doing so, we will to a large extent follow the approach of \cite[\S 2]{KP}, where the authors considered the stable case of localizing pairs.

	However, it is necessary to point out that the form of the classification obtained by this approach is in a sense weaker than Theorem~\ref{T:bounded} or \cite{AJS}. While we are able to show that any compactly generated t-structure corresponds to a Thomason filtration and is generated by suspensions of Koszul complexes, we will not show in general that the aisle is determined cohomology-wise by supports. Indeed, such a result seems to be unknown to the literature even for localizing pairs. For commutative noetherian rings, this stronger description is obtained in \cite{AJS}. Another known case is the localizing pair associated to a ``basic'' Thomason set $P = V(I)$ based on a single finitely generated ideal $I$ (\cite{N}, \cite{DG}, \cite{Pos}). We provide a version of the latter result for aisles in \ref{S:localizations}, which will also allow for a description of the coaisle of a compactly generated t-structure via \v{C}ech cohomology. 
	
	Given an object $X \in \Der(R)$, by a \EMP{non-negative suspension} we mean an object isomorphic in $\Der(R)$ to $X[i]$ for some $i \geq 0$.
	\begin{lem}\label{L:recipe}
			Suppose that $\mathcal{S}$ is a set of compact objects, and let $C$ be a compact object. If $C \in \pa(\mathcal{S})$, then $C$ is quasi-isomorphic to a direct summand of an object $F$, where $F$ is an $n$-fold extension of finite direct sums of non-negative suspensions of objects of $\mathcal{S}$ for some $n > 0$.
	\end{lem}
	\begin{proof}
			Using \cite[Theorem 12.3]{Kel}, we know that $C \simeq \dhocolim_{n \geq 0} X_n$, where $X_n$ is an $n$-fold extension of coproducts of non-negative suspensions of objects from $\mathcal{S}$. This means that there is for each $n > 0$ a cardinal $\Lambda_n$ and a sequence $(S_\lambda^n \mid \lambda < \Lambda_n)$ of non-negative suspensions of objects from the set $\mathcal{S}$ such that there is a triangle 
			$$X_{n-1} \rightarrow X_n \rightarrow \bigoplus_{\lambda < \Lambda_n} S_\lambda^n \rightarrow X_{n-1}[1].$$
			Since $C$ is compact, we can use \cite[Proposition 3.13]{R} to infer that there is an integer $n > 0$, and an object $F$, which is obtained as an $n$-fold extension of objects of form $\bigoplus_{\lambda \in A_k}S_\lambda^k$ for $k=1,\ldots,n$, where $A_k$ is a finite subset of $\Lambda_k$, and such that $C$ is quasi-isomorphic to a direct summand of $F$. This concludes the proof.
	\end{proof}
	Recall that a complex $X$ is \EMP{finite} if the $R$-module $\bigoplus_{n \in \mathbb{Z}}H^n(X)$ is finitely generated. The fact that the generation of aisles by finite complexes over a commutative noetherian ring is controlled by the support of cohomologies follows from a result due to Kiessling in the following way.
	\begin{prop}\label{P:acyclic}
			Let $R$ be a commutative noetherian ring, and let $X$ and $Y$ be two finite complexes over $R$. Suppose that for every $k \in \mathbb{Z}$ we have that
			$$\Supp(H^k(Y)) \subseteq \Supp(\bigoplus_{i \geq k}H^i(X)).$$
			Then $\aisle(Y) \subseteq \aisle(X)$.
	\end{prop}
	\begin{proof}
			This follows from \cite[Corollary 7.4]{Ki}. Indeed, \cite[Corollary 7.4]{Ki} claims that the condition on supports above implies that $Y$ is contained in the smallest subcategory of $\Der(R)$ containing $X$ and closed under extensions, cones, and coproducts. Then \emph{a fortiori}, we have $Y \in \aisle(X)$ as desired.
	\end{proof}
	\begin{lem}\label{L:noeth}
			Let $R$ be a commutative noetherian ring, and $S \in \Der^c(R)$. Then there is $n > 0$, ideals $I_1,I_2,\ldots,I_n$ of $R$ and integers $s_1,\ldots,s_n$ such that $\pa(S) = \pa(R/I_1[s_1],\ldots,R/I_n[s_n])$.
	\end{lem}
	\begin{proof}
			Because $R$ is commutative noetherian, the complex $C$ is finite, as well as the split complex $Y = \bigoplus_{k \in \mathbb{Z}} R/\Ann(H^k(C))[-k]$. Since $\Supp(H^k(C)) = V(\Ann H^k(C))$, we can use Proposition~\ref{P:acyclic} to infer that $\aisle(C) = \aisle(Y)$. Because $C$ is compact, only finitely many of its cohomologies do not vanish, and therefore $\aisle(C) = \aisle(R/I_1[s_1],\cdots,R/I_n[s_n])$ for ideals $I_1,\ldots,I_n$ and appropriate integers $s_1,\ldots,s_n$.
	\end{proof}
	In preparation for the main result, we need to prove some auxiliary observations about aisles generated by cyclic modules and Koszul complexes over any commutative ring.
\begin{lem}\label{L:generation}
		Let $R$ be a commutative ring, and $I$ be a finitely generated ideal.
		\begin{enumerate}
			\item[(i)] $\aisle(R/I) = \aisle(R/I^n)$ for any $n > 0$,
			\item[(ii)] $M \in \aisle(R/I)$ whenever $M$ is an $R$-module such that $\Supp(M) \subseteq V(I)$,
			\item[(iii)] $\aisle(K(I)) = \aisle(R/I)$.
		\end{enumerate}
	\end{lem}
	\begin{proof}
	\item[(i)] First, for any ideal $J$, any $R/J$-module $M$ is quasi-isomorphic to a complex of free $R/J$-modules concentrated in non-positive degrees. Using brutal truncations and Milnor colimits, we see that such $M$ belongs to $\aisle(R/J)$. In particular, $R/I \in \aisle(R/I^n)$. On the other hand, as $R/I^n$ admits a finite filtration by $R/I$-modules, $R/I^n \in \aisle(R/I)$.
	\item[(ii)] Since $\Supp(M) \subseteq V(I)$, $M$ is in the hereditary torsion class corresponding to the Thomason set $V(I)$. Therefore, there is an epimorphism $\bigoplus_{n > 0} (R/I^n) ^{(\varkappa_n)} \rightarrow M$ for some cardinals $\varkappa_n$. Also, the kernel $K$ of this epimorphism satisfies $\Supp(K) \subseteq V(I)$ again. Inductively, we can construct a complex concentrated in non-positive degrees with coordinates consisting of direct sums of copies of cyclic modules $R/I^n, n > 0$. By $(i)$, and using the Milnor colimits of brutal truncations again, this shows that $M \in \aisle(R/I)$.
	\item[(iii)] By \cite[Proposition 1.6.5b]{BH}, $\Supp(H^n(K(I))) \subseteq V(I)$ for all $n \in \mathbb{Z}$. Since $K(I)$ is a bounded complex, $K(I)$ can be obtained by a finite number of extensions from stalks of modules supported on $V(I)$ concentrated in non-positive degrees. This shows that $K(I) \in \aisle(R/I)$ by $(ii)$. 

		We know by \cite[p. 360, 8.2.7]{NC} that $H^0(K(I)) \simeq R/I$. Using Proposition~\ref{P:rigid}, $K(I) \otimes_R R/I \in \aisle(K(I))$. Inspecting the definition of a Koszul complex, $K(I) \otimes_R R/I$ is a split complex of $R/I$-modules, and $H^0(K(I) \otimes_R R/I) \simeq R/I$. Therefore, $R/I \in \aisle(K(I))$.
	\end{proof}
	The following proof is a version of the argument \cite[Proposition 2.1.13]{KP} adjusted for aisles.
\begin{lem}\label{L:replacement}
			Let $R$ be an arbitrary commutative ring, and $S \in \Der^c(R)$. Then there is $n > 0$, ideals $I_1,I_2,\ldots,I_n$ of $R$ and integers $s_1,\ldots,s_n$ such that $\pa(S) = \pa(K(I_1)[s_1],\ldots,K(I_n)[s_n])$.
	\end{lem}
	\begin{proof}
			First, by \cite[Lemma A.2]{N}, there is a noetherian subring $T$ of $R$, and a compact object $S' \in \Der^c(T)$ such that $S' \otimes_T R = S$. Using Lemma~\ref{L:noeth} together with Lemma~\ref{L:generation}(iii), there are ideals $J_1,\ldots,J_n$ and integers $s_1,\ldots,s_n$ such that $\aisle(S') = \aisle(K(J_1)[s_1],\ldots,K(J_n)[s_n])$ in $\Der(T)$. By Lemma~\ref{L:recipe}, $K(J_i)[s_i]$ can be obtained as a direct summand in an object $F_i \in \Der^c(T)$, where $F_i$ is an $n_i$-fold extension of finite direct sums of non-negative suspensions of copies of $S'$, for each $i=1,\ldots,n$ and for some $n_i > 0$. Applying $- \otimes_T^\mathbf{L} R$ onto $F_i$, we conclude that $K(J_i)[s_i] \otimes_T R$ is a direct summand in $F_i \otimes_T R$, an $n_i$-fold extension of finite direct sums of non-negative suspensions of copies of $S = S' \otimes_T R$ (note that by compactness of all of the involved objects, we can compute the derived tensor product via the ordinary tensor product). One can see directly from the definition of a Koszul complex that, when choosing the appropriate generating sets, $K(J_i) \otimes_T R$ is isomorphic to a Koszul complex $K(I_i)$, where $I_i = J_iR$ for each $i=1,\ldots,n$.

		We proved that $K(I_i)[s_i] \in \aisle(S)$ in $\Der(R)$ for each $i=1,\ldots,n$. Repeating the same argument using Lemma~\ref{L:recipe} with the roles of the two aisles reversed shows that also $S \in \aisle(K(I_1)[s_1],\ldots,K(I_n)[s_n])$.
	\end{proof}
	
	\begin{thm}\label{T:full}
		Let $R$ be a commutative ring. Then there is a 1-1 correspondence:
		$$\left \{ \begin{tabular}{ccc} \text{ Thomason filtrations $\Phi$} \\ \text{of $\Spec(R)$} \end{tabular}\right \}  \leftrightarrow  \left \{ \begin{tabular}{ccc} \text{Aisles $\mathcal{U}$ of compactly generated } \\ \text{t-structures in $\Der(R)$} \end{tabular}\right \}.$$
		The correspondence is given by mutually inverse assignments
		$$A: \Phi \mapsto \mathcal{U} = \aisle(K(I)[-n] \mid V(I) \subseteq \Phi(n) ~\forall n \in \mathbb{Z}\},$$
		and
		$$B: \mathcal{U} \mapsto \Phi(n) = \bigcup \{V(I) \mid I ~\text{ f.g. ideal such that } R/I[-n] \in \mathcal{U}\}.$$
	\end{thm}
	\begin{proof}
		Since both the assignments $A$ and $B$ are clearly well-defined, it is enough to check that they are mutually inverse. 

			First, let $\mathcal{U}$ be a compactly generated aisle. By Lemma~\ref{L:replacement}, we know that $\mathcal{U} = \aisle(\mathcal{S})$, where $\mathcal{S}$ is a set of non-negative suspensions of Koszul complexes. Let us prove that $\mathcal{U} = AB(\mathcal{U})$. Clearly, $\mathcal{U} \subseteq AB(\mathcal{U})$. To prove the converse inclusion, let $J$ be a finitely generated ideal such that $V(J) \subseteq \Phi(n)$, where $\Phi = B(\mathcal{U})$. Since $\Phi(n) = \bigcup \{V(I) \mid I ~\text{ f.g. ideal such that } R/I[-n] \in \mathcal{U}\}$, this means that there are finitely generated ideals $I_1,\ldots,I_k$, such that $I_1^{l_1}\cdots I_k^{l_k} \subseteq J$, for some natural numbers $k,l_1,\ldots,l_k$, and such that $R/I_j[-n] \in \mathcal{U}$ for all $j=1,\ldots,k$. Because the cyclic module $R/I_1^{l_1}\cdots I_k^{l_k}$ admits a finite filtration by modules from hereditary torsion classes $\mathcal{T}_i = \{M \in \ModR \mid \Supp(M) \subseteq V(I_i)\}$, $i=1,\ldots,k$, we infer that $R/I_1^{l_1}\cdots I_k^{l_k} [-n]$ belongs to $\mathcal{U}$ by Lemma~\ref{L:generation}(ii). By the same lemma, $R/J[-n] \in \mathcal{U}$. Therefore, for any finitely generated ideal $J$ with $V(J) \subseteq \Phi(n)$ we have $\aisle(K(J)[-n]) = \aisle(R/J[-n]) \subseteq \mathcal{U}$. This shows that $\mathcal{U} = AB(\mathcal{U})$.

			Now let $\Phi$ be a Thomason filtration, and let us prove that $BA(\Phi) = \Phi$. Again, by Lemma~\ref{L:generation}(iii), clearly $BA(\Phi)(n) \supseteq \Phi(n)$ for each $n \in \mathbb{Z}$. To prove the other inclusion, let $J$ be a finitely generated ideal such that $R/J[-n] \in A(\Phi)$. Because $\mathcal{U}_\Phi = \{X \in \Der(R) \mid \Supp(H^j(X)) \subseteq \Phi(j)\ ~\forall j \in \mathbb{Z}\}$ is a subcategory of $\Der(R)$ closed under extensions, suspensions, and coproducts, and $\Supp(H^j(K(I)) \subseteq V(I)$ for any finitely generated ideal $I$ and any $j \in \mathbb{Z}$, we have $A(\Phi) \subseteq \mathcal{U}_\Phi$. Then $R/J[-n] \in \mathcal{U}_\Phi$, and therefore $V(J) \subseteq \Phi(n)$.
	\end{proof}
\subsection{Compactly generated co-t-structures}
Following \cite[Theorem 4.10]{SP}, we know that there is a 1-1 correspondence in $\Der(R)$ between the compactly generated t-structures and the compactly generated co-t-structures in $\Der(R)$. Explicitly, the correspondence can be described as follows. If $S \in \Der^c(R)$ is an compact object, we define the \EMP{compact-dual} of $S$ to be $S^*:=\RHom_R(S,R) \simeq \Hom_R(S,R)$. Given a t-structure $(\mathcal{U},\mathcal{V})$ generated by a set $\mathcal{S}$ of compact objects, define a torsion pair $(\mathcal{X},\mathcal{Y})$ generated by the set $\mathcal{S}^* = \{S^* \mid S \in \mathcal{S}\}$ of compact-duals of $\mathcal{S}$. By \cite{SP}, the pair $(\mathcal{X},\mathcal{Y})$ is a co-t-structure, and this assignment is the desired correspondence.

Starting with a Thomason filtration $\Phi$, we define a set 
$$\mathcal{S}_\Phi = \{K(I)[-n] \mid ~\forall V(I) \subseteq \Phi(n) ~\forall n \in \mathbb{Z}\}$$
of compact objects. We can express the compactly generated coaisle $\mathcal{V}_\Phi$ associated to $\Phi$ via Theorem~\ref{T:full} as follows, using \cite[15.68.0.2]{Stacks}:
$$\mathcal{V}_\Phi = \mathcal{S}_\Phi\Perp{0} = \{ X \in \Der(R) \mid \Hom_{\Der(R)}(K(I)[-i],X)=0 ~\forall i \in \mathbb{Z} ~\forall V(I) \subseteq \Phi(i)\} = $$
$$= \{ X \in \Der(R) \mid \Hom_{\Der(R)}(K(I),X[k])=0 ~\forall k \leq i \in \mathbb{Z} ~\forall V(I) \subseteq \Phi(i)\} =$$ 
$$= \{ X \in \Der(R) \mid \RHom_R(K(I),X) \in \Der{}^{>i} ~\forall i \in \mathbb{Z} ~\forall V(I) \subseteq \Phi(i)\}.$$
If $(\mathcal{X},\mathcal{Y}_\Phi)$ is the compactly generated co-t-structure corresponding to the t-structure $(\mathcal{U},\mathcal{V}_\Phi)$ in the above-described manner, we claim that
\begin{equation}\label{E:cot}
	\mathcal{Y}_\Phi = \{ X \in \Der(R) \mid K(I) \otimes_R^{\mathbf{L}} X \in \Der{}^{<-i} ~\forall i \in \mathbb{Z} ~\forall V(I) \subseteq \Phi(i)\}.
\end{equation}
To establish (\ref{E:cot}), we proceed as follows. First, for any $X \in \Der(R)$ we have:
$$X \in \mathcal{Y}_\Phi = (\mathcal{S}_\Phi^*)\Perp{0} \iff \Hom_{\Der(R)}((K(I)[-i])^*,X) = 0 ~\forall i \in \mathbb{Z} ~\forall V(I) \subseteq \Phi(i).$$
Furthermore,
$$(K(I)[-i])^* \simeq K(I)^*[i],$$
whence, by similar computation as above,
$$\mathcal{Y}_\Phi = \{ X \in \Der(R) \mid \RHom_R(K(I)^*, X) \in \Der{}^{<-i} ~\forall i \in \mathbb{Z} ~\forall V(I) \subseteq \Phi(i)\}.$$
By \cite[Lemma 20.43.11]{Stacks}, we can use the compactness of $K(I)$ to infer that
$$\RHom_R(K(I)^*, X) \simeq K(I)^{**} \otimes_R^\mathbf{L} X \simeq K(I) \otimes_R^\mathbf{L} X.$$
This proves (\ref{E:cot}). Altogether, this yields the following classification result for compactly generated co-t-structures:
\begin{thm}\label{T:dualmain}
	Let $R$ be a commutative ring. Then there is a 1-1 correspondence:
		$$\left \{ \begin{tabular}{ccc} \text{ Thomason filtrations $\Phi$} \\ \text{of $\Spec(R)$} \end{tabular}\right \}  \leftrightarrow  \left \{ \begin{tabular}{ccc} \text{Compactly generated } \\ \text{co-t-structures in $\Der(R)$} \end{tabular}\right \}.$$
				The correspondence is given by the assignment 
				$$\Phi \mapsto (\Perp{0}\mathcal{Y}_\Phi, \mathcal{Y}_\Phi).$$ 
\end{thm}
\subsection{Compactly generated t-structures versus localizing pairs}\label{S:localizations}
	The purpose of this subsection is twofold --- to describe the coaisles of the compactly generated t-structures explicitly, and to make precise the relation between the classification of compactly generated localizing subcategories (via Thomason sets), and the present classification of compactly generated t-structures (via $\mathbb{Z}$-filtrations by Thomason sets).

	A triangulated subcategory $\mathcal{L}$ of $\Der(R)$ is called \EMP{localizing} if it is closed under coproducts. If $(\mathcal{L},\mathcal{L}\Perp{0})$ forms a localizing pair, that is, if the inclusion $\mathcal{L} \subseteq \Der(R)$ admits a right adjoint, we call $\mathcal{L}$ a \EMP{strict localizing} subcategory. If in this situation also the class $\mathcal{L}\Perp{0}$ is localizing, we call $(\mathcal{L},\mathcal{L}\Perp{0})$ a \EMP{smashing} localizing pair, and the class $\mathcal{L}$ a \EMP{smashing} subcategory. A localizing pair is smashing if and only if it is homotopically smashing, this follows from \cite[\S 5.2]{MU} (and it is not true for t-structures in general, see \cite[Example 6.2]{SSV}).

	The term ``smashing'' comes from algebraic topology and indicates that the smashing subcategory is determined by the symmetric monoidal product, which in our category $\Der(R)$ means the derived tensor product. Indeed, that is the case. Let $\mathcal{L}$ be a smashing subcategory and consider the left (resp. right) approximation functor $\Gamma$ (resp. $L$) with respect to the localizing pair $(\mathcal{L},\mathcal{L}\Perp{0})$. Then the approximation triangle
	$$\Gamma(R) \rightarrow R \rightarrow L(R) \rightarrow \Gamma(R)[1],$$
	is an idempotent triangle in the sense of \cite{BF}, that is, $\Gamma(R) \otimes_R^\mathbf{L} L(R) = 0$, and by \cite[Theorem 2.13]{BF} we have:
	$$L \simeq (L(R) \otimes_R^\mathbf{L} - ), ~ ~ ~ \Gamma \simeq (\Gamma(R) \otimes_R^\mathbf{L} -).$$
	In particular, $\mathcal{L} = \Ker(L(R) \otimes_R^\mathbf{L} -)$, and $\mathcal{L}\Perp{0} = \Ker(\Gamma(R) \otimes_R^\mathbf{L} -)$. 	
	
	It is clear that any compactly generated localizing pair is smashing. The question whether the converse is also true is known as the telescope conjecture, and it is not valid in general, but it is true for example for commutative noetherian rings \cite{N}, or for (even one-sided) hereditary rings \cite{KS}.

	The classification of compactly generated localizing pairs is due to Thomason \cite{T}, who generalized the result of Neeman and Hopkins from noetherian rings to arbitrary commutative rings. It also follows as a special case of Theorem~\ref{T:full}, as compactly generated localizing pairs amongst compactly generated t-structures are clearly precisely those, such that the corresponding Thomason filtration is constant. We recall that given a set $\mathcal{S} \subseteq \Der(R)$, $\Loc(\mathcal{S})$ denotes the smallest localizing subcategory in $\Der(R)$ containing $\mathcal{S}$. Equivalently, $\Loc(\mathcal{S}) = \aisle(\mathcal{S}[n] \mid n<0)$, and therefore $\Loc(\mathcal{S})$ is always a strict localizing subcategory.
	\begin{thm}\label{T:thomason}\emph{(\cite{T}, \cite{KP}}
			Let $R$ be a commutative ring. There is a 1-1 correspondence 
				$$\left \{ \begin{tabular}{ccc} \text{ Thomason subsets $P$} \\ \text{of $\Spec(R)$} \end{tabular}\right \}  \leftrightarrow  \left \{ \begin{tabular}{ccc} \text{Compactly generated} \\ \text{localizing subcategories $\mathcal{L}$ of $\Der(R)$} \end{tabular}\right \},$$
				given by the assignment
				$$X \mapsto \mathcal{L}_X = \Loc(K(I) \mid V(I) \subseteq P).$$
	\end{thm}
	If the Thomason set $P$ is equal to $V(I)$ for some finitely generated ideal $I$, then the induced idempotent triangle is given by a \v{C}ech complex, as we explain now. Given an element $x \in R$, we let 
	$$\check{C}^{\sim}(x) = (\cdots \rightarrow 0 \rightarrow R \xrightarrow{\iota} R[x^{-1}] \rightarrow 0 \rightarrow \cdots),$$
	a complex concentrated in degrees 0 and 1, where $\iota$ is the obvious natural map. For an $n$-tuple $\bar{x}=(x_1,x_2,\ldots,x_n)$ of elements of $R$ we let $\check{C}^{\sim}(\bar{x}) = \bigotimes_{i=1}^n \check{C}^{\sim}(x_i)$. It follows from \cite[Corollary 3.12]{Gre} that if $\bar{x}$ and $\bar{y}$ are two finite sequences of generators of an ideal $I$, then $\check{C}^{\sim}(\bar{x})$ is quasi-isomorphic to $\check{C}^{\sim}(\bar{y})$. Therefore, given a finitely generated ideal $I$, we can use any finite generating sequence of $I$ to define $\check{C}^{\sim}(I)$ as an object in $\Der(R)$, and call it the \EMP{infinite Koszul complex}\footnote{Also known as the ``stable Koszul complex'' in the literature.} associated to $I$. Let $I$ be generated by a finite sequence $\bar{x}$. Note that, up to quasi-isomorphism, $\check{C}^{\sim}(I)$ is a complex of form
	$$R \rightarrow \bigoplus_{1 \leq i_1 \leq n} R[x_{i_1}^{-1}] \rightarrow \bigoplus_{1 \leq i_1 < i_2 \leq n} R[x_{i_1}^{-1},x_{i_2}^{-1}] \rightarrow \cdots \rightarrow R[x_1^{-1},x_2^{-1},\ldots,x_n^{-1}].$$ 
	Taking the cone of the map $\check{C}^{\sim}(I) \rightarrow R$, given by identity in degree 0, results in a triangle
	\begin{equation}\label{E:Cech}
		\Delta_I: \check{C}^{\sim}(I) \rightarrow R \rightarrow \check{C}(I) \rightarrow \check{C}^{\sim}(I)[1].
	\end{equation}
	We call the complex $\check{C}(I)$ the \EMP{\v{C}ech complex} associated to the finitely generated ideal $I$. It follows e.g. from \cite[Lemma 1.1]{Pos} that $\check{C}^{\sim}(I) \in \mathcal{L}_{V(I)}$. Because $\mathcal{L}\Perp{0}$ is closed under extensions, tensoring by arbitrary complexes (\cite[Lemma 1.1.8]{KP}), and contains all stalk complexes $R[x_1^{-1}],\ldots,R[x_n^{-1}]$ (\cite[Theorem 2.2.4]{KP}), also $\check{C}(I) \in \mathcal{L}_{V(I)}\Perp{0}$. Then the triangle (\ref{E:Cech}) is the approximation triangle of $R$ with respect to the localizing pair $(\mathcal{L}_{V(I)},\mathcal{L}_{V(I)}\Perp{0})$. In particular, 
	$$\mathcal{L}_{V(I)} = \Ker(\check{C}(I) \otimes_R -) \text{   and   } \mathcal{L}_{V(I)}\Perp{0} = \Ker(\check{C}^{\sim}(I) \otimes_R -)$$ 
	(note that, as both $\check{C}(I)$ and $\check{C}^{\sim}(I)$ are bounded complexes of flat $R$-modules, we can drop the left derivation symbol from the tensor product).

	Finally, we explain how the stable Koszul complex is built from the (compact) Koszul complexes. One can see directly from the construction that for a single element $x \in R$, the stable Koszul complex $\check{C}^{\sim}(x)$ is obtained as a direct limit of compact-duals of Koszul complexes $K(x^n), n>0$: 
	$$\check{C}^{\sim}(x) \simeq \varinjlim_{n>0} K(x^n)^*.$$ 
	If $x_1,x_2,\ldots,x_m$ is a sequence of generators of an ideal $I$, it is then easy to that 
	$$
			\check{C}^{\sim}(I) \simeq \otimes_{i=1}^m \check{C}^{\sim}(x_i) \simeq \otimes_{i=1}^m \varinjlim_{n>0} K(x_i^n)^* \simeq \varinjlim_{n>0} \otimes_{i=1}^m K(x_i^n)^* \simeq \varinjlim_{n>0} K(x_1^n,\ldots,x_m^n)^*.
$$
\begin{lem}\label{L:adjoint}
		Let $I$ be a finitely generated ideal of a commutative ring $R$. Then $\aisle(K(I)) = \mathcal{L}_{V(I)} \cap \Der^{\leq 0} = \{X \in \Der^{\leq 0} \mid \Supp(X) \subseteq V(I) ~\forall n \leq 0\}$. Furthermore, the functor defined as the composition $\tau^{\leq 0} \circ (\check{C}^{\sim}(I) \otimes_R -)$ is the right adjoint to the inclusion of $\aisle(K(I))$ into $\Der(R)$. In particular, $\aisle(K(I))\Perp{0} = \{Y \in \Der(R) \mid \check{C}^{\sim}(I) \otimes_R Y \in \Der^{>0}\}$.

\end{lem}
\begin{proof}
		First, recall that both $\tau^{\leq 0}$ and $\check{C}^{\sim}(I)$ are the right adjoints to the inclusions of $\Der^{\leq 0}$ and $\mathcal{L}_{V(I)}$ into $\Der(R)$, respectively. Moreover, we have by \cite[Proposition 5.1]{Pos} that $\mathcal{L}_{V(I)} = \{X \in \Der(R) \mid \Supp(H^n(X)) \subseteq V(I) ~\forall n \in \mathbb{Z}\}$. Because the essential image of the composition $\tau^{\leq 0} \circ (\check{C}^{\sim}(I) \otimes_R -)$ is precisely $\mathcal{L}_{V(I)} \cap \Der^{\leq 0}$, we infer that $\tau^{\leq 0} \circ (\check{C}^{\sim}(I) \otimes_R -)$ is the right adjoint to the inclusion of $\mathcal{L}_{V(I)} \cap \Der^{\leq 0} = \{X \in \Der^{\leq 0} \mid \Supp(H^n(X)) \subseteq V(I) ~\forall n \leq 0\}$ to $\Der(R)$.

		Next we show that $\mathcal{L}_{V(I)} \cap \Der^{\leq 0} = \aisle(K(I))$. Since $\Supp(H^n(K(I))) = V(I)$ for all $n \in \mathbb{Z}$, we clearly have $\aisle(K(I)) \subseteq \mathcal{L}_{V(I)} \cap \Der^{\leq 0}$. In order to show the other inclusion, let $Y$ be a complex from $\mathcal{L}_{V(I)} \cap \Der^{\leq 0}$, and consider the approximation triangle with respect to the t-structure $(\aisle(K(I)),\aisle(K(I))\Perp{0})$:
		$$U \rightarrow Y \rightarrow X \rightarrow X[1].$$ 
		Since both $U$ and $Y$ are in $\mathcal{L}_{V(I)} \cap \Der^{\leq 0}$, so is the cone $X$. We will show that $X=0$, and therefore $Y \simeq U \in \aisle(K(I))$. Let $x_1,\ldots,x_m$ be a finite sequence of generators of $I$, and let $K_n = K(x_1^n,\ldots,x_m^n)$ denote the Koszul complex of $n$-th powers of the generators for each $n>0$. Using Lemma~\ref{L:generation}, we see that $\aisle(K(I)) = \aisle(K_n)$ for all $n > 0$, and therefore $K_n \in \aisle(K(I))$ for all $n>0$. Then we can compute using the discussion above and \cite[Lemma 20.43.11]{Stacks}:
		$$X \simeq X \otimes_R \check{C}^{\sim}(I) \simeq X \otimes_R \varinjlim_{n>0} K_n^* \simeq \varinjlim_{n>0} (X \otimes_R K_n^*) \simeq \dhocolim_{n>0} \mathbf{R}\Hom_R(K_n,X).$$
		Because $X \in \aisle(K(I))\Perp{0}$, we see that $\mathbf{R}\Hom_R(K_n,X) \in \Der^{>0}$ for all $n>0$, and therefore also $X \simeq \dhocolim_{n>0} \mathbf{R}\Hom_R(K_n,X) \in \Der^{>0}$. Since $X \in \Der^{\leq 0}$, we conclude that $X=0$. We proved that $\mathcal{L}_{V(I)} \cap \Der^{\leq 0} = \aisle(K(I))$.

		Since the left approximation functor associated to $\aisle(K(I))$ is $\tau^{\leq 0} \circ (\check{C}^{\sim}(I) \otimes_R -)$, it follows that the coaisle $\aisle(K(I))\Perp{0}$ is equal to $\{Y \in \Der(R) \mid \tau^{\leq 0} (\check{C}^{\sim}(I) \otimes_R Y) = 0\} = \{Y \in \Der(R) \mid \check{C}^{\sim}(I) \otimes_R Y \in \Der^{>0}\}$.
\end{proof}

\begin{prop}\label{P:Cech}
		Let $R$ be a commutative ring and $\Phi$ a Thomason filtration. Consider the compactly generated t-structure associated to $\Phi$ via Theorem~\ref{T:full}, that is, a t-structure $(\mathcal{U},\mathcal{V}_\Phi)$ generated by the set $\mathcal{S}_\Phi = \{K(I)[-n] \mid ~\forall V(I) \subseteq \Phi(n) ~\forall n \in \mathbb{Z}\}$. Then the coaisle $\mathcal{V}_\Phi$ can be described as follows:
			$$\mathcal{V}_\Phi = \{Y \in \Der(R) \mid (\check{C}^{\sim}(I) \otimes_R Y) \in \Der{}^{>n} ~\forall V(I) \subseteq \Phi(n) ~\forall n \in \mathbb{Z}\}.$$
	\begin{proof}
			Clearly, we have
						\begin{equation}\label{E:intersection}
					\mathcal{V}_\Phi = \mathcal{S}_\Phi\Perp{0}= \bigcap_{n \in \mathbb{Z}, V(I) \subseteq \Phi(n)} \aisle(K(I)[-n])\Perp{0}.
			\end{equation}
			Using Lemma~\ref{L:adjoint} and shifting, we have 
			$$ \aisle(K(I)[-n])\Perp{0}= \{Y \in \Der(R) \mid (\check{C}^{\sim}(I) \otimes_R Y) \in \Der{}^{>n}\}.$$ Combining this with (\ref{E:intersection}) yields the desired description of $\mathcal{V}_\Phi$.
	\end{proof}
\end{prop}

\section{Intermediate t-structures and cosilting complexes}\label{S:cosilting}
	We finish the paper by discussing the consequences of Theorem~\ref{T:TC} and Theorem~\ref{T:bounded} for the silting theory of a commutative noetherian ring. An object $C \in \Der(R)$ is a \EMP{(bounded) cosilting complex} if the two following conditions hold: 
	\begin{enumerate}
			\item $C$ belongs to $\Htpy^b(\operatorname{Inj-R})$, that is, $C$ is quasi-isomorphic to a bounded complex of injective $R$-modules, and,
			\item the pair $(\Perp{\leq 0}C, \Perp{>0}C)$ forms a t-structure (note that this implies, in particular, that $C \in \Perp{>0}C$)\footnote{The author is grateful to Rosanna Laking for pointing out the unnecessity of this condition in the definition of a cosilting object to him.}.
	\end{enumerate}
	The adjective ``bounded'' in this context reads as condition (1), and variants of cosilting object not satisfying condition (1) are also discussed in literature, especially in the setting of a general compactly generated triangulated category (we refer to \cite{NSZ}, \cite{PV}). In this paper, all cosilting complexes are bounded. The notion of a cosilting complex was introduced in \cite{ZW} as a dualization of the (big) silting complexes of \cite{AMV0}. Combining the recent works \cite{MV} and \cite{L} gives a useful characterization of t-structures induced by cosilting complexes. Before that, we need to recall a couple of definitions. Say that a t-structure $(\mathcal{U},\mathcal{V})$ is \EMP{intermediate} if there are integers $n \leq m$ such that $\Der^{\geq m} \subseteq \mathcal{V} \subseteq \Der^{\geq n}$. A \EMP{TTF triple} is a triple $(\mathcal{U},\mathcal{V},\mathcal{W})$ of subcategories of $\Der(R)$ such that both $(\mathcal{U},\mathcal{V})$ and $(\mathcal{V},\mathcal{W})$ are torsion pairs in $\Der(R)$.
	\begin{thm}\label{T:cosilting}
			Let $R$ be a (not necessarily commutative) ring, and $(\mathcal{U},\mathcal{V})$ an intermediate t-structure in $\Der(R)$. Then the following conditions are equivalent:
			\begin{enumerate}
					\item[(i)] $(\mathcal{U},\mathcal{V})$ is induced by a (bounded) cosilting complex,
					\item[(ii)] there is a (cosuspended) TTF triple $(\mathcal{U},\mathcal{V},\mathcal{W})$,
					\item[(iii)] $(\mathcal{U},\mathcal{V})$ is homotopically smashing.
			\end{enumerate}
	\end{thm}
	\begin{proof}
			The equivalence $(i) \Leftrightarrow (ii)$ is precisely \cite[Theorem 3.13]{MV}. Also, \cite[Theorem 3.14]{MV} implies that $(i)$ is equivalent to $\mathcal{V}$ being \EMP{definable} (see \cite{L} for the definition), which is in our situation equivalent to $(\mathcal{U},\mathcal{V})$ being homotopically smashing by \cite[Theorem 4.6]{L}.
	\end{proof}
	Using Theorem~\ref{T:TC}, we conclude with the following:
	\begin{cor}\label{C02}
			If the ring $R$ is commutative noetherian, the conditions of Theorem~\ref{T:cosilting} are equivalent to 
			\begin{enumerate}
				\item[(iv)] $(\mathcal{U},\mathcal{V})$ is compactly generated.
			\end{enumerate}
	\end{cor}
	Corollary~\ref{C02} can be seen as a generalization of the cofinite type of $n$-cotilting modules over commutative noetherian ring proved in \cite{APST} to the cosilting setting. Also, in view of condition (ii) from Theorem~\ref{T:cosilting}, it provides a partial answer to \cite[Question 3.12]{SP}. Let us call a Thomason filtration $\Phi$ \EMP{bounded} if there are integers $m<n$ such that $\Phi(m) = \Spec(R)$ and $\Phi(n) = \emptyset$. We will also adopt the custom from cotilting theory and say that two cosilting complexes are \EMP{equivalent} if they induce the same cosilting t-structure.
	\begin{thm}\label{T:cosilt}
		Let $R$ be a commutative noetherian ring. There is a bijective correspondence:
		$$\left \{ \begin{tabular}{ccc} \text{ Bounded Thomason} \\ \text{filtrations $\Phi$ of $\Spec(R)$} \end{tabular}\right \}  \leftrightarrow  \left \{ \begin{tabular}{ccc} \text{Cosilting complexes} \\ \text{in $\Der(R)$ up to equivalence} \end{tabular}\right \}.$$
		The cosilting t-structure $(\mathcal{U}_\Phi,\mathcal{V}_\Phi)$ induced by a cosilting module associated to a Thomason filtration $\Phi$ by this correspondence can be described as follows:
		$$\mathcal{U}_\Phi = \{X \in \Der(R) \mid \Supp H^n(X) \subseteq \Phi(n) ~\forall n \in \mathbb{Z}\},$$
			$$\mathcal{V}_\Phi = \{X \in \Der(R) \mid \check{C}^{\sim}(I) \otimes_R X \in \Der{}^{>n} ~\forall V(I) \subseteq \Phi(n) ~\forall n \in \mathbb{Z}\}.$$
	\end{thm}
	\begin{proof}
			It is easy to see that bounded Thomason filtration correspond via Theorem~\ref{T:bounded} precisely to those compactly generated t-structures which are intermediate. The rest is a consequence of Corollary~\ref{C02}, Theorem~\ref{T:bounded}, and Proposition~\ref{P:Cech}.
	\end{proof}
	We conclude the paper by remarking that Theorem~\ref{T:cosilt} is a generalization of \cite[Theorem 5.1]{AH}, which gives an equivalence between cosilting modules and Thomason sets over a commutative noetherian ring $R$. Cosilting modules were introduced in \cite{AMV0} as module-theoretic counterparts of 2-term cosilting complexes in $\Der(R)$, and indeed, they correspond in a well-behaved manner precisely to 2-term cosilting complexes, up to a shift (\cite[Theorem 4.11]{AMV0}). Theorem~\ref{T:cosilt} extends this result from 2-term complexes to cosilting complexes of arbitrary length.
	\section*{Acknowledgement}
The paper was written during the author's stay at Dipartimento di Matematica of Universit\`{a} degli Studi di Padova. I would like to express my gratitude to everybody at the department --- and to Prof. Silvana Bazzoni in particular --- for all the hospitality, and all the stimulating discussions. Also, I am indebted to Jan \v{S}\v{t}ov\'{i}\v{c}ek for spotting an error in an earlier version of the manuscript.
\bibliographystyle{alpha}
\bibliography{comp_gen}

\begin{thebibliography}{ATLJSS03}

\bibitem[AF12]{AF}
Frank~W Anderson and Kent~R Fuller.
\newblock {\em Rings and categories of modules}, volume~13.
\newblock Springer Science \& Business Media, 2012.

\bibitem[AHH16]{AH}
Lidia Angeleri~H{\"u}gel and Michal Hrbek.
\newblock Silting modules over commutative rings.
\newblock {\em International Mathematics Research Notices},
  2017(13):4131--4151, 2016.

\bibitem[AHMV15]{AMV0}
Lidia Angeleri~H{\"u}gel, Frederik Marks, and Jorge Vit{\'o}ria.
\newblock Silting modules.
\newblock {\em International Mathematics Research Notices}, 2016(4):1251--1284,
  2015.

\bibitem[AHMV17]{AMV}
Lidia Angeleri~H{\"u}gel, Frederik Marks, and Jorge Vit{\'o}ria.
\newblock Torsion pairs in silting theory.
\newblock {\em Pacific Journal of Mathematics}, 291(2):257--278, 2017.

\bibitem[AHP{\v{S}}T14]{APST}
Lidia Angeleri~H{\"u}gel, David Posp{\'\i}{\v{s}}il, Jan
  {\v{S}}\v{t}ov{\'\i}{\v{c}}ek, and Jan Trlifaj.
\newblock Tilting, cotilting, and spectra of commutative noetherian rings.
\newblock {\em Transactions of the American Mathematical Society},
  366(7):3487--3517, 2014.

\bibitem[ATLJS10]{AJS}
Leovigildo Alonso~Tarr{\'\i}o, Ana L{\'o}pez~Jerem{\'\i}as, and Manuel
  Saor{\'\i}n.
\newblock Compactly generated t-structures on the derived category of a
  noetherian ring.
\newblock {\em Journal of Algebra}, 324(3):313--346, 2010.

\bibitem[ATLJSS03]{AJS2}
Leovigildo Alonso~Tarr{\'\i}o, Ana L{\'o}pez~Jerem{\'\i}as, and
  Mar{\'\i}a~Jos{\'e} Souto~Salorio.
\newblock Construction of t-structures and equivalences of derived categories.
\newblock {\em Transactions of the American Mathematical Society},
  355(6):2523--2543, 2003.

\bibitem[BBD82]{BBD}
Alexander~A. Beilinson, Joseph Bernstein, and Pierre Deligne.
\newblock Faisceaux pervers.
\newblock In {\em Analysis and topology on singular spaces, {I} ({L}uminy,
  1981)}, volume 100 of {\em Ast\'erisque}, pages 5--171. Soc. Math. France,
  Paris, 1982.

\bibitem[BF11]{BF}
Paul Balmer and Giordano Favi.
\newblock Generalized tensor idempotents and the telescope conjecture.
\newblock {\em Proceedings of the London Mathematical Society},
  102(6):1161--1185, 2011.

\bibitem[BH93]{BH}
Winfried Bruns and J{\"u}rgen Herzog.
\newblock {\em Cohen-Macaulay rings, volume 39 of Cambridge studies in advanced
  mathematics}.
\newblock Cambridge University Press, Cambridge, 1993.

\bibitem[BN93]{NB}
Marcel B{\"o}kstedt and Amnon Neeman.
\newblock Homotopy limits in triangulated categories.
\newblock {\em Compositio Mathematica}, 86(2):209--234, 1993.

\bibitem[DG02]{DG}
William~G Dwyer and John Patrick~Campbell Greenlees.
\newblock Complete modules and torsion modules.
\newblock {\em American Journal of Mathematics}, 124(1):199--220, 2002.

\bibitem[GP08]{GP}
Grigory Garkusha and Mike Prest.
\newblock Torsion classes of finite type and spectra.
\newblock {\em K-theory and Noncommutative Geometry}, pages 393--412, 2008.

\bibitem[Gre07]{Gre}
J.~P.~C. Greenlees.
\newblock First steps in brave new commutative algebra.
\newblock In {\em Interactions between homotopy theory and algebra}, volume 436
  of {\em Contemp. Math.}, pages 239--275. Amer. Math. Soc., Providence, RI,
  2007.

\bibitem[GT12]{GT}
R{\"u}diger G{\"o}bel and Jan Trlifaj.
\newblock {\em Approximations and Endomorphism Algebras of Modules: Volume
  1--Approximations/Volume 2--Predictions}, volume~41.
\newblock Walter de Gruyter, 2012.

\bibitem[Hop87]{Ho}
Michael~J. Hopkins.
\newblock Global methods in homotopy theory.
\newblock In {\em Proceedings of the 1985 LMS Symposium on Homotopy Theory},
  pages 73--96, 1987.

\bibitem[H{\v{S}}17]{HS}
Michal Hrbek and Jan {\v{S}}\v{t}ov{\'\i}{\v{c}}ek.
\newblock Tilting classes over commutative rings.
\newblock {\em arXiv preprint arXiv:1701.05534}, 2017.

\bibitem[Kel94]{Kel2}
Bernhard Keller.
\newblock A remark on the generalized smashing conjecture.
\newblock {\em Manuscripta Mathematica}, 84(1):193--198, 1994.

\bibitem[Kie12]{Ki}
Jonas Kiessling.
\newblock Properties of cellular classes of chain complexes.
\newblock {\em Israel Journal of Mathematics}, 191(1):483--505, 2012.

\bibitem[KN12]{Kel}
Bernhard Keller and Pedro Nicol{\'a}s.
\newblock Weight structures and simple dg modules for positive dg algebras.
\newblock {\em International Mathematics Research Notices}, 2013(5):1028--1078,
  2012.

\bibitem[KP17]{KP}
Joachim Kock and Wolfgang Pitsch.
\newblock Hochster duality in derived categories and point-free reconstruction
  of schemes.
\newblock {\em Transactions of the American Mathematical Society},
  369(1):223--261, 2017.

\bibitem[KS05]{KP2}
Masaki Kashiwara and Pierre Schapira.
\newblock {\em Categories and sheaves}, volume 332.
\newblock Springer Science \& Business Media, 2005.

\bibitem[Kv10]{KS}
Henning Krause and Jan \v{S}\v{t}ov\'{i}\v{c}ek.
\newblock The telescope conjecture for hereditary rings via {E}xt-orthogonal
  pairs.
\newblock {\em Advances in Mathematics}, 225(5):2341--2364, 2010.

\bibitem[Lak18]{L}
Rosanna Laking.
\newblock Purity in compactly generated derivators and t-structures with
  {G}rothendieck hearts.
\newblock {\em arXiv preprint arXiv:1804.01326}, 2018.

\bibitem[Laz69]{Laz}
Daniel Lazard.
\newblock Autour de la platitude.
\newblock {\em Bull. Soc. Math. France}, 97(81):128, 1969.

\bibitem[Mat58]{M}
Eben Matlis.
\newblock Injective modules over noetherian rings.
\newblock {\em Pacific Journal of Mathematics}, 8(3):511--528, 1958.

\bibitem[Mur06]{MU}
Daniel Murfet.
\newblock Derived categories part {I}.
\newblock \url{http://therisingsea.org/notes/DerivedCategories.pdf}, 2006.

\bibitem[MV18]{MV}
Frederik Marks and Jorge Vit{\'o}ria.
\newblock Silting and cosilting classes in derived categories.
\newblock {\em Journal of Algebra}, 501:526--544, 2018.

\bibitem[NB92]{N}
Amnon Neeman and Marcel B{\"o}kstedt.
\newblock The chromatic tower for \textbf{D}({R}).
\newblock {\em Topology}, 31(3):519--532, 1992.

\bibitem[Nee00]{N2}
Amnon Neeman.
\newblock Oddball {B}ousfield classes.
\newblock {\em Topology}, 39(5):931--935, 2000.

\bibitem[Nor68]{NC}
Douglas~Geoffrey Northcott.
\newblock Lessons on rings, modules and multiplicities.
\newblock 1968.

\bibitem[NSZ18]{NSZ}
Pedro Nicol{\'a}s, Manuel Saor{\'\i}n, and Alexandra Zvonareva.
\newblock Silting theory in triangulated categories with coproducts.
\newblock {\em Journal of Pure and Applied Algebra}, 2018.

\bibitem[Pos16]{Pos}
Leonid Positselski.
\newblock Dedualizing complexes and {MGM} duality.
\newblock {\em Journal of Pure and Applied Algebra}, 220(12):3866--3909, 2016.

\bibitem[PV18]{PV}
Chrysostomos Psaroudakis and Jorge Vit{\'o}ria.
\newblock Realisation functors in tilting theory.
\newblock {\em Mathematische Zeitschrift}, 288(3-4):965--1028, 2018.

\bibitem[Rou08]{R}
Rapha{\"e}l Rouquier.
\newblock Dimensions of triangulated categories.
\newblock {\em Journal of K-theory}, 1(2):193--256, 2008.

\bibitem[{\v{S}}P16]{SP}
Jan {\v{S}}\v{t}ov{\'\i}{\v{c}}ek and David Posp{\'\i}{\v{s}}il.
\newblock On compactly generated torsion pairs and the classification of
  co-t-structures for commutative noetherian rings.
\newblock {\em Transactions of the American Mathematical Society},
  368(9):6325--6361, 2016.

\bibitem[S{\v{S}}V17]{SSV}
Manuel Saor{\'\i}n, Jan {\v{S}}\v{t}ov{\'\i}{\v{c}}ek, and Simone Virili.
\newblock {t}-structures on stable derivators and {G}rothendieck hearts.
\newblock {\em arXiv preprint arXiv:1708.07540}, 2017.

\bibitem[{Sta}18]{Stacks}
The {Stacks Project Authors}.
\newblock \textit{Stacks Project}.
\newblock \url{http://stacks.math.columbia.edu}, 2018.

\bibitem[{\v{S}}\v14]{Sto}
Jan {\v{S}}\v{t}ov{\'\i}{\v{c}}ek.
\newblock Derived equivalences induced by big cotilting modules.
\newblock {\em Advances in Mathematics}, 263:45--87, 2014.

\bibitem[Tho97]{T}
Robert~W. Thomason.
\newblock The classification of triangulated subcategories.
\newblock {\em Compositio Mathematica}, 105(1):1--27, 1997.

\bibitem[ZW17]{ZW}
Peiyu Zhang and Jiaqun Wei.
\newblock Cosilting complexes and {AIR}-cotilting modules.
\newblock {\em Journal of Algebra}, 491:1--31, 2017.

\end{thebibliography}
\end{document}